\documentclass[11pt]{amsart}
\usepackage[numbers,sort&compress]{natbib}
\usepackage{float}
\usepackage{tikz}
\usepackage{amsmath, amsthm, color, fullpage}
\usepackage{graphicx, amsaddr}
\usepackage{enumerate}
\usepackage{epsfig}
\usepackage{epstopdf}
\usepackage{tikz}
\usepackage{array,amsmath, enumerate,  url, psfrag}
\usepackage{amssymb, amsaddr, fullpage}

\usepackage{tikz}
\usetikzlibrary{shapes.geometric, positioning}

\usepackage{graphicx,subfigure}
\usepackage{color}
\newtheorem{theorem}{Theorem}[section]

\newtheorem{lemma}[theorem]{Lemma}

\newtheorem{definition}[theorem]{Definition}

\title{Flexible DP-$4$-coloring of planar graphs without $4$-cycles and intersecting triangles}
\author{Shu Fang$^{1}$, Runrun Liu$^{1}$, Gexin Yu$^{2,*}$}

\address{
$^{1}$\small School of Mathematics, Zhejiang Normal University, Jinhua, Zhejiang 321004, China.\\
$^{2}$\small Department of Mathematics, William \& Mary, Williamsburg, VA 23185, USA.
}

\thanks{*Corresponding author. Email: gyu@wm.edu}

\thanks{The research of the second author was supported by ZJNSFC (No.\,LJHSZ26A010001) and NSFC (No.\,12101563, 12371360).}

\begin{document}
\maketitle

\begin{abstract}
\baselineskip=14pt
Graph coloring with preferences offers a powerful framework for constraint satisfaction problems in which fulfilling every request is impossible but satisfying a guaranteed positive fraction is highly desirable. A \emph{request} on a graph $G$ equipped with a list assignment $L$ assigns to each vertex of some subset $dom(r)\subseteq V(G)$ a preferred color from its list. Following Dvo\v{r}\'{a}k, Norin, and Postle (2019), $G$ is \emph{$\varepsilon$-flexibly $k$-choosable} if, for every $k$-list assignment $L$ and every request $r$, there is an $L$-coloring of $G$ that agrees with $r$ on at least $\varepsilon|dom(r)|$ vertices. The corresponding notion for DP-coloring (correspondence coloring) was formalized by Bradshaw, Choi, and Kostochka (2025). Choi, Clemen, Ferrara, Horn, Ma, and Masa\v{r}\'{i}k (2022) proved that every planar graph without $4$-cycles and with $3$-cycle distance at least $2$ is $\varepsilon$-flexibly $4$-choosable. We improve the result in two respects: weakening the hypothesis from $3$-cycle distance $\geq 2$ to vertex-disjoint triangles, and strengthening the conclusion from list flexibility to weighted DP-flexibility:
\emph{Every simple planar graph without $4$-cycles and without intersecting triangles is weighted $\varepsilon$-flexibly DP-$4$-colorable.}
The list size $4$ is sharp: Montassier, Raspaud, and Wang constructed a planar graph without $4$-cycles, $5$-cycles, and intersecting triangles that is not $3$-choosable.
\end{abstract}
\baselineskip=16pt

\section{\bf Introduction}

The flexibility framework of Dvo\v{r}\'{a}k, Norin, and Postle~\cite{DNP19} interpolates between list coloring and precoloring extension. Rather than asking whether a partial precoloring extends to a full $L$-coloring, which is typically a hard question with a binary answer, one asks for an $L$-coloring that honors a constant fraction of a set of vertex-color preferences. This is the natural relaxation when preferences come from soft constraints, and it has driven a productive line of research in recent years.

\subsection*{List flexibility}

A \emph{list assignment} for a graph $G$ is a function $L$ that assigns to each $v\in V(G)$ a set $L(v)$ of colors; an \emph{$L$-coloring} is a proper coloring $\phi$ with $\phi(v)\in L(v)$ for all $v$. The graph $G$ is \emph{$k$-choosable} if it is $L$-colorable for every $L$ with $|L(v)|\geq k$. A \emph{request} for $(G,L)$ is a function $r$ with domain $dom(r)\subseteq V(G)$ and $r(v)\in L(v)$. For $\varepsilon>0$, $r$ is \emph{$\varepsilon$-satisfiable} if there is an $L$-coloring $\phi$ of $G$ with $\phi(v)=r(v)$ for at least $\varepsilon|dom(r)|$ vertices $v\in dom(r)$. We say $(G,L)$ is \emph{$\varepsilon$-flexible} if every request is $\varepsilon$-satisfiable, and $G$ is \emph{$\varepsilon$-flexibly $k$-choosable} if $(G,L)$ is $\varepsilon$-flexible for every $k$-list assignment $L$.

A \emph{weighted request} is a function $w:\{(v,c):v\in V(G),\,c\in L(v)\}\to\mathbb{R}_{\geq 0}$. Writing $w(G,L)=\sum_{v,c}w(v,c)$, we call $w$ to be \emph{$\varepsilon$-satisfiable} if some $L$-coloring $\phi$ satisfies
$$\sum_{v\in V(G)}w(v,\phi(v))\;\geq\;\varepsilon\,w(G,L).$$
Then $G$ is \emph{weighted $\varepsilon$-flexibly $k$-choosable} if every weighted request from every $k$-list assignment is $\varepsilon$-satisfiable.

\subsection*{Prior results for sparse and planar graphs}

The study of flexibility on sparse graphs has attracted considerable attention. Dvo\v{r}\'{a}k--Norin--Postle~\cite{DNP19} established that $d$-degenerate graphs are weighted $\varepsilon$-flexibly $(d+2)$-choosable,  conjectured that such graphs are $\varepsilon$-flexibly $(d+1)$-choosable, and verified the conjecture for graphs with maximum average degree less than $d+\frac{2}{d+3}$ . They also proved that planar graphs (which are $5$-degenerate) are $\varepsilon$-flexibly $6$-choosable. Kaul, Mathew, Mudrock, and Pelsmajer~\cite{KMMP24} improved the existential $\varepsilon$ bound of~\cite{DNP19} by giving explicit estimates for the fraction of satisfiable requests.  Bradshaw, Masa\v{r}\'{i}k, Stacho~\cite{BMS22} characterized the graphs of maximum degree $\Delta$ that that $\varepsilon$-flexibly $\Delta$-choosable, which answered a question of~\cite{DNP19}. %Cambie, Cames van Batenburg, and Zhu~\cite{CCZ23+} established a connection between list packing and the flexibility framework. 
Subsequent research focused on reducing the list size:
\begin{itemize}
\item Dvo\v{r}\'{a}k, Masa\v{r}\'{i}k, Mus\'{i}lek, and Pangr\'{a}c~\cite{DMMP20,DMMP21}: $\varepsilon$-flexibly $3$-choosable for planar graphs with girth $\geq 6$; $\varepsilon$-flexibly $4$-choosable for triangle-free planar graphs.
\item Bi and Bradshaw~\cite{BB23+}: $\varepsilon$-flexibly $3$-choosable for graphs with $mad(G)<3$, and $4$-choosable for graphs with $mad(G)<\frac{11}{3}$.
\item Masa\v{r}\'{i}k~\cite{M19}: $\varepsilon$-flexibly $5$-choosable for planar without $C_4$, a $4$-cycle.
\item Choi, Clemen, Ferrara, Horn, Ma, and Masa\v{r}\'{i}k~\cite{CCFHMM22}: $\varepsilon$-flexibly $5$-choosable for planar without $K_4^-$; \emph{$\varepsilon$-flexibly $4$-choosable for planar without $C_4$ with $3$-cycle distance $\geq 2$}; $\varepsilon$-flexibly $4$-choosable for planar without $\{C_4,C_5,C_6\}$.
\item Yang and Yang~\cite{YY26}: $\varepsilon$-flexibly $4$-choosable for planar without $\{C_4,C_5\}$.
\item Yang~\cite{Y22}: more sufficient conditions for $\varepsilon$-flexibly $5$-choosability.
\end{itemize}

\subsection*{DP-flexibility}

DP-coloring (correspondence coloring), introduced by Dvo\v{r}\'{a}k and Postle~\cite{DP18}, generalizes list coloring. Given $G$ and $h:V(G)\to\mathbb{N}$, an \emph{$h$-cover} $(H,L)$ of $G$ is a graph $H$ with
\begin{itemize}
\item a clique $L(v)\subseteq V(H)$ of size $h(v)$ for each $v\in V(G)$,
\item edges between $L(u)$ and $L(v)$ forming a matching whenever $uv\in E(G)$,
\item no edges between $L(u)$ and $L(v)$ whenever $uv\notin E(G)$.
\end{itemize}
An \emph{$(H,L)$-coloring} of $G$ is a map $\phi:V(G)\to V(H)$ with $\phi(v)\in L(v)$ for all $v$ and $\phi(x)\phi(y)\notin E(H)$ for all $xy\in E(G)$. The graph $G$ is \emph{DP-$k$-colorable} if every $k$-cover admits an $(H,L)$-coloring.

A \emph{weighted DP-request} is a function $w:V(H)\to\mathbb{R}_{\geq 0}$. We say $(G,H,L)$ is \emph{weighted $\varepsilon$-flexible} if for every such $w$ there is an $(H,L)$-coloring $\phi$ with
$$\sum_{v\in V(G)}w(\phi(v))\;\geq\;\varepsilon\sum_{v\in V(G)}\sum_{c\in L(v)}w(c),$$
and $G$ is \emph{weighted $\varepsilon$-flexibly DP-$k$-colorable} if this holds for every $k$-cover.

Bradshaw, Choi, and Kostochka~\cite{BCK25+} showed that sparse multigraphs with maximum average degree $<3$ are $\varepsilon$-flexibly DP-$3$-colorable, extending the list flexibility results by Bi and Bradshaw~\cite{BB23+}. Cambie, Cames van Batenburg and Zhu~\cite{CCZ23+} considered flexible choosability and flexible DP-colorings under packing list-colorings and packing DP-colorings, and obtained some very interesting results. 

\subsection*{Main result}

Lam, Xu, and Liu~\cite{LXL99} and Wang and Lih~\cite{WL02} proved $4$-choosability for the planar subclasses excluding $C_4$ and excluding intersecting triangles, respectively. The following natural flexibility question remains open: does list size $4$ suffice for $\varepsilon$-flexibility under \emph{either} hypothesis alone? Under the conjunction, \cite{CCFHMM22} obtained $\varepsilon$-flexibility at list size $4$, but only under the stronger requirement that triangles lie at distance $\geq 2$ (which forbids triangles even when joined by an edge). We weaken this to vertex-disjointness and simultaneously upgrade to weighted DP-flexibility.

\begin{theorem}\label{main}
There exists $\varepsilon>0$ such that every simple planar graph without $4$-cycles and without intersecting triangles is weighted $\varepsilon$-flexibly DP-$4$-colorable.
\end{theorem}

Two triangles are \emph{intersecting} if they share a vertex; equivalently, the hypothesis says that the distance between $3$-cycles is at least $1$.

The list size $4$ cannot be reduced: Montassier, Raspaud, and Wang~\cite{MRW06} constructed a planar graph without $4$-cycles, $5$-cycles, and intersecting triangles that is not $3$-choosable. So Theorem~\ref{main} is best possible with respect to the list size.

\section{\bf Preliminaries}

Here we fix the terminology used throughout the paper. For a vertex $v$, we call $v$ a {\em $k$-vertex}, a {\em $k^+$-vertex}, or a {\em $k^-$-vertex} if $d(v)=k$, $d(v)\ge k$, or $d(v)\le k$, respectively. Similarly, we define a {\em $k$-face}, a {\em $k^+$-face}, and a {\em $k^-$-face}, and likewise $k$-, $k^+$-, and $k^-$-paths and cycles.
A $k$-face $f=[v_1v_2\cdots v_k]$ is an \emph{$(x_1,x_2,\ldots,x_k)$-face} if $d(v_i)=x_i$ for each $i\in[k]$. If a $k$-vertex $v$ is adjacent to a vertex $u$, then we call $v$ a \emph{$k$-neighbor} of $u$; similarly, we define a \emph{$k^+$-neighbor} and a \emph{$k^-$-neighbor}. A face $f$ is a \emph{pendent $3$-face} of a vertex $v$ if $v$ is not on $f$ but is adjacent to some $3$-vertex $u$ on $f$; in this case we call $u$ a \emph{pendent $3$-neighbor} of $v$. Let $dist(u,v)$ denote the distance between $u$ and $v$ in $G$.

Given a function $f: V(S)\to \mathbb{Z}$ and a vertex $v\in V(S)$, let $f\downarrow v$ denote the function with $(f\downarrow v)(w)=f(w)$ for $w\ne v$ and $(f\downarrow v)(v)=1$. A list assignment $L$ is an \emph{$f$-assignment} if $|L(v)|\ge f(v)$ for all $v\in V(S)$. Given a set of graphs $\mathcal{F}$ and a graph $H$, a set $I\subset V(S)$ is \emph{$\mathcal{F}$-forbidding} if the graph $H$ together with one additional vertex adjacent to all vertices in $I$ contains no graph from $\mathcal{F}$. Let $S$ be an induced subgraph of $G$ and $v\in V(S)$. We define $d_{S}(v)$ as the degree of $v$ in the induced subgraph $S$. The following two definitions and two lemmas are from~\cite{CCFHMM22} in terms of list coloring; we restate them in the setting of DP-coloring.

\begin{definition}(\cite{CCFHMM22})\label{d1}
Let $G$ be a graph with a $k$-cover $(H,L)$. A graph $S$ is an \emph{$(\mathcal{F},k)$-boundary-reducible induced
subgraph} of $G$ if there exists a set $B\subsetneq V(S)$ such that
\begin{itemize}
\item[(FIX)] for every $v\in V(S)\setminus B$, $S-B$ is $(H,L)$-colorable for every $((k-d_G+d_{(S-B)})\downarrow v)$-assignment $L$, and
\item[(FORB)] for every $\mathcal{F}$-forbidding set $I\subseteq V(S)\setminus B$ of size at most $k-2$, $S-B$ is $(H,L)$-colorable for every $(k-d_G+d_{(S-B)}-\mathbf{1}_I)$-assignment $L$.
\end{itemize}
\end{definition}

\begin{definition}(\cite{CCFHMM22})\label{d2}
Let $G$ be a graph with a $k$-cover $(H,L)$ that contains no graph in $\mathcal{F}$ as an induced subgraph. An \emph{$(\mathcal{F},k,b)$-resolution} of $G$ is a sequence of nested subgraphs $G_i$ for $0\le i\le M$ such that $G_0:=G$ and
$$G_i:=G-\bigcup_{j=1}^i(H_j-B_j),$$
where each $H_i$ is an induced $(\mathcal{F},k)$-boundary-reducible induced subgraph of $G_{i-1}$ with boundary $B_i$ satisfying $|V(H_i)\setminus B_i|\le b$, and $G_M$ is an $(\mathcal{F},k)$-boundary-reducible induced subgraph with empty boundary and size at most $b$. For technical reasons, set $G_{M+1}:=\emptyset$.
 \end{definition}

Our goal is to show that every graph containing no subgraph from $\mathcal{F}$ contains a boundary-reducible induced subgraph. Conceptually, we then think of a resolution as an inductively defined object obtained by iteratively identifying a reducible subgraph with boundary $B$ and deleting $S-B$ until $V(G)$ is exhausted.

The following two lemmas are essentially as in~\cite{DNP19} and~\cite{CCFHMM22}; we include the proofs here for DP-coloring.

\begin{lemma}(\cite{DNP19} for list flexibility)\label{DP1}
Let $G$ be a graph with a $k$-cover $(H,L)$. Suppose $G$ is DP-$k$-colorable and there exists a probability distribution on $(H,L)$-colorings $\phi$ of $G$ such that $\operatorname{Prob}[\phi(v)=c]\ge\varepsilon$ for every $v\in V(G)$ and $c\in L(v)$. Then $G$ is $\varepsilon$-flexibly DP-$k$-colorable.
\end{lemma}

\begin{proof}
Let $w$ be a weighted request for $H$, and let $\phi$ be chosen at random from the postulated probability distribution. By linearity of expectation,
\[
\mathrm{E}\left[\sum_{v\in V(G)}w(\phi(v))\right] = \sum_{v\in V(G),\,c\in L(v)}\operatorname{Prob}[\phi(v) = c]\cdot w(c)\geq \varepsilon\sum_{v\in V(G)}\sum_{c\in L(v)}w(c),
\]
and thus there exists an $L$-coloring $\phi$ with $\sum_{v\in V(G)}w(\phi(v))\ge\varepsilon\sum_{v\in V(G)}\sum_{c\in L(v)}w(c)$, as required.
\end{proof}

\begin{lemma}(\cite{CCFHMM22} for list flexibility)\label{DP2}
For all integers $k\ge3$ and $b\ge1$ and for every set $\mathcal{F}$ of forbidden subgraphs there exists an $\varepsilon>0$ as follows. Let $G$ be a graph with an $(\mathcal{F},k,b)$-resolution. Then $G$ is $\varepsilon$-flexibly DP-$k$-colorable.
\end{lemma}
\begin{proof}
Let $p = k^{-b}$ and $\varepsilon = p^{k-1}$.
Let $G$ be a graph satisfying the assumptions and let $(H, L)$ be a $k$-cover of $G$.
We prove the following claim by induction on the $(\mathcal{F}, k, b)$-resolution; part~(i) implies that $G$ is $\varepsilon$-flexibly DP-$k$-colorable by Lemma~\ref{DP1}.

\smallskip
\noindent\emph{Claim.} There exists a probability distribution on $(H, L)$-colorings $\phi$ of $G$ such that
\begin{enumerate}
\item[(i)] for every $v \in V(G)$ and every $c \in L(v)$, $\operatorname{Prob}[\phi(v) = c] \geq \varepsilon$, and
\item[(ii)] for every color $c$ and every $\mathcal{F}$-forbidding set $I$ in $G$ of size at most $k-2$, $\operatorname{Prob}[\phi(v) \neq c] \geq p^{|I|}$ for all $v\in I$.
\end{enumerate}

\smallskip
The claim holds trivially for a graph with no vertices, which is the base case.
Hence, suppose $V(G) \neq \emptyset$.
By the assumptions, there exists an induced subgraph $S$ of $G$ that is $(\mathcal{F}, k)$-boundary-reducible.
Let $B \subset V(S)$ be its boundary and let $Q = S - B$.
By definition, $|V(Q)| \leq b$.

By the induction hypothesis, there exists a probability distribution on $(H, L)$-colorings of $G-(S-B)$ satisfying (i) and (ii).
Choose an $(H, L)$-coloring $\psi$ from this distribution.
Define a new cover $(H', L')$ on $G[V(Q)]$ as follows:
for each $y \in V(Q)$, let $L'(y)$ be the set of colors $c \in L(y)$ such that there is no edge in $H$ between $c$ and $\psi(v)$ for any neighbor $v \in V(G) \setminus V(Q)$ of $y$.
Then $|L'(y)| \geq k - d_G(y) + d_{G[Q]}(y)$ for all $y \in V(Q)$.
By (FORB) with $I = \emptyset$, $G[V(Q)]$ admits an $(H', L')$-coloring.
Choose one uniformly at random among all such colorings and extend $\psi$ to an $(H, L)$-coloring $\phi$ of $G$.

We first verify (ii).
Let $I_1 = I \setminus V(Q)$ and $I_2 = I \cap V(Q)$.
By the induction hypothesis, $\operatorname{Prob}[\phi(v) \neq c] \geq p^{|I_1|}$ for all $v \in I_1$.
If $I_2 = \emptyset$, we are done.
Otherwise, for $y \in I_2$ define $L_c(y) = L'(y) \setminus \{c'\}$, where $c$ is matched to $c'\in L'(y)$ in $H$; for $y \in V(Q) \setminus I_2$, let $L_c(y) = L'(y)$.
Then $|L_c(y)| \geq k - d_G(y) + d_{G[Q]}(y) - \mathbf{1}_I(y)$.
By (FORB), $G[V(Q)]$ has an $(H_Q, L_c)$-coloring, where $H_Q$ is the subgraph of $H$ induced by the vertices in $Q$.
Since $G[V(Q)]$ has at most $k^{b}$ such colorings, the probability that $\phi(y) \neq c$ for all $y \in I_2$ is at least $1/k^{b} = p \geq p^{|I_2|}$.
Thus the probability that $\phi(v) \neq c$ for all $v \in I$ is at least $p^{|I_1| + |I_2|} \geq p^{|I|}$, proving (ii).

Now we prove (i).
If $v \in V(G) \setminus V(Q)$, the claim follows from the induction hypothesis.
Assume $v \in V(Q)$.
Let $I$ be the set of neighbors of $v$ in $V(G) \setminus V(Q)$.
Since $G$ contains no graph from $\mathcal{F}$ and $B \subseteq V(G) \setminus V(Q)$, the set $I$ is $\mathcal{F}$-forbidding in $G - V(Q)$.
Moreover, (FORB) implies $|I| \leq k-2$.
By the induction hypothesis, $\operatorname{Prob}[\psi(u) \neq c] \geq p^{k-2}$ for all $u \in I$.
Assuming this is the case, (FIX) implies there exists an $(H_Q, L')$-coloring of $G[V(Q)]$ with $\phi(v) = c$.
Since there are at most $k^{b}$ such colorings, the probability that $\phi(v) = c$ is at least $p^{k-2} / k^{b} = \varepsilon$.
Hence (i) holds.
\end{proof}

\section{\bf Properties of minimal counterexamples}
Let $\mathcal{F}$ consist of $4$-cycles and intersecting triangles. Let $G$ be a counterexample to Theorem~\ref{main} with minimum $|V(G)|$. Fix a plane embedding of $G$, and note that by minimality $G$ is connected. Let $(H,L)$ be a $4$-cover of $G$. We will show that the following configurations cannot appear in $G$:
\begin{itemize}
\item[(RC1)] A $2^-$-vertex.
\item[(RC2)] A $3$-face with two $3$-vertices.
\item[(RC3)] A $5$-face $f=[v_1v_2v_3v_4v_5]$ with $d(v_1)=d(v_4)=3$, $d(v_2)=d(v_3)=4$, and each of $v_2,v_3$ on a $3$-face (which could be the same for both).
\item[(RC4)] A $3^+$-vertex $v$ with at least $d(v)-1$ neighbors of degree $3$.
\item[(RC5)] A $3^+$-vertex $v$ on a $3$-face with at least $d(v)-2$ neighbors of degree $3$.
\item[(RC6)] A $3^+$-vertex $v$ with at least $d(v)-2$ pendent $3$-faces.
\item[(RC7)] A $3$-path $P=v_1v_2v_3$ with $d(v_1)\ge3$, $d(v_2)\ge3$, $d(v_3)\ge4$, and $v_i$ having at least $d(v_i)-3$ pendent $3$-neighbors in $N(v_i)\setminus V(P)$ for each $i\in[3]$.
\item[(RC8)] A $4$-path $v_1v_2v_3v_4$ with $d(v_1),d(v_3),d(v_4)\ge3$, $d(v_2)=4$, and $v_2$ on a $3$-face, where for each $i\in\{1,3,4\}$, $v_i$ has at least $d(v_i)-3$ pendent $3$-neighbors in $N(v_i)\setminus V(P)$.
\item[(RC9)] A poor $5$-face $f$ weakly incident to no special $5^+$-vertices. (the definitions appear later)
\end{itemize}

\begin{figure}[htbp]
	\centering

	% (RC2)
	\begin{minipage}[b]{0.31\textwidth}
		\centering
		\begin{tikzpicture}[scale=.65]
			\begin{scope}[xshift=4cm]			
				\draw (0,0) node[circle, draw=black, fill=white, inner sep=0mm, minimum size=1.6mm, label={right: $w$}] (w){};
				\draw (-1,-1.732) node[circle, draw=black, fill=black, inner sep=0mm, minimum size=1.6mm, label={left: $u$}](u){};	
				\draw (1,-1.732) node[circle, draw=black, fill=black, inner sep=0mm, minimum size=1.6mm, label={right: $v$}] (v){};
				\draw (w)--(u)--(v)--(w);
				\draw(w) -- (-0.5,0.866);
				\draw(w) -- (-0.2,0.980);
				\draw(w) -- (0.2,0.980);
				\draw(w) -- (0.5,0.866);
				\draw(u) -- (-1.5,-2.598);
				\draw(v) -- (1.5,-2.598);
				\draw (0,-4) node {(RC2)};
			\end{scope}
		\end{tikzpicture}
	\end{minipage}
	\hfill
	% (RC3)
	\begin{minipage}[b]{0.31\textwidth}
		\centering
		\begin{tikzpicture}[scale=.7]		
				\draw (0,0) node[circle, draw=black, fill=white, inner sep=0mm, minimum size=1.6mm, label={right: $v_5$}] (v_5){};
				\draw (1.214,-0.882) node[circle, draw=black, fill=black, inner sep=0mm, minimum size=1.6mm, label={right: $v_1$}](v_1){};	
				\draw (0.750,-2.308) node[circle, draw=black, fill=black, inner sep=0mm, minimum size=1.6mm, label={right: $v_2$}] (v_2){};
				\draw (-0.750,-2.308) node[circle, draw=black, fill=black, inner sep=0mm, minimum size=1.6mm, label={left: $v_3$}] (v_3){};
				\draw (-1.214,-0.882) node[circle, draw=black, fill=black, inner sep=0mm, minimum size=1.6mm, label={left: $v_4$}](v_4){};	
				\draw (v_5)--(v_1)--(v_2)--(v_3)--(v_4)--(v_5);
				\draw(v_5) -- (-0.375,0.650);
				\draw(v_5) -- (-0.15,0.735);
				\draw(v_5) -- (0.15,0.735);
				\draw(v_5) -- (0.375,0.650);
				\draw(v_1) -- (1.821,-0.441);
				\draw(v_4) -- (-1.821,-0.441);
				\draw(v_2) -- (1.66,-2.90);
				\draw(v_2) -- (1.04,-3.36);
				\draw(1.66,-2.90) -- (1.04,-3.36);
				\draw(v_3) -- (-1.66,-2.90);
				\draw(v_3) -- (-1.04,-3.36);
				\draw(-1.66,-2.90) -- (-1.04,-3.36);
				\draw (0,-4) node {(RC3)};
		\end{tikzpicture}
	\end{minipage}
	\hfill
	% (RC4)
	\begin{minipage}[b]{0.31\textwidth}
		\centering
		\begin{tikzpicture}[scale=.55]	
				\draw (0,0) node[circle, draw=black, fill=black, inner sep=0mm, minimum size=1.6mm, label={right: $v$}] (v){};
				\draw (1,-1.732) node[circle, draw=black, fill=black, inner sep=0mm, minimum size=1.6mm, label={}](v_1){};	
				\draw (2.5,-1.732) node[circle, draw=black, fill=black, inner sep=0mm, minimum size=1.6mm, label={}] (v_2){};
				\draw (-1,-1.732) node[circle, draw=black, fill=black, inner sep=0mm, minimum size=1.6mm, label={}](v_3){};	
				\draw (-2.5,-1.732) node[circle, draw=black, fill=black, inner sep=0mm, minimum size=1.6mm, label={}] (v_4){};
				\draw (v)--(v_1);
				\draw (v)--(v_2);
				\draw (v)--(v_3);
				\draw (v)--(v_4);
				\draw (v)--(0,1.5);
				\draw (v_1)--(0.5,-3);
				\draw (v_1)--(1.5,-3);
				\draw (v_2)--(2,-3);
				\draw (v_2)--(3,-3);
				\draw (v_3)--(-0.5,-3);
				\draw (v_3)--(-1.5,-3);
				\draw (v_4)--(-2,-3);
				\draw (v_4)--(-3,-3);
				\draw (0, -1.732) node {$\dots$};
				\draw (0,-4) node {(RC4)};
		\end{tikzpicture}
	\end{minipage}

\vspace{1.2em}
% (RC5)
\begin{minipage}[b]{0.31\textwidth}
	\centering
	\begin{tikzpicture}[scale=.55]			
			\draw (0,0) node[circle, draw=black, fill=black, inner sep=0mm, minimum size=1.6mm, label={right: $v$}] (v){};
			\draw (1,-1.732) node[circle, draw=black, fill=black, inner sep=0mm, minimum size=1.6mm, label={}](v_1){};	
			\draw (2.5,-1.732) node[circle, draw=black, fill=black, inner sep=0mm, minimum size=1.6mm, label={}] (v_2){};
			\draw (-1,-1.732) node[circle, draw=black, fill=black, inner sep=0mm, minimum size=1.6mm, label={}](v_3){};	
			\draw (-2.5,-1.732) node[circle, draw=black, fill=black, inner sep=0mm, minimum size=1.6mm, label={}] (v_4){};
			\draw (v)--(v_1);
			\draw (v)--(v_2);
			\draw (v)--(v_3);
			\draw (v)--(v_4);
			\draw (v)--(0.9,1.4);
			\draw (v)--(-0.9,1.4);
			\draw (0.9,1.4)--(-0.9,1.4);
			\draw (v_1)--(0.5,-3);
			\draw (v_1)--(1.5,-3);
			\draw (v_2)--(2,-3);
			\draw (v_2)--(3,-3);
			\draw (v_3)--(-0.5,-3);
			\draw (v_3)--(-1.5,-3);
			\draw (v_4)--(-2,-3);
			\draw (v_4)--(-3,-3);
			\draw (0, -1.732) node {$\dots$};
			\draw (0,-4) node {(RC5)};
	\end{tikzpicture}
\end{minipage}
\hfill
% (RC6)
\begin{minipage}[b]{0.31\textwidth}
	\centering
	\begin{tikzpicture}[scale=.55]
			\draw (0,0) node[circle, draw=black, fill=black, inner sep=0mm, minimum size=1.6mm, label={right: $v$}] (v){};
			\draw (1,-1.732) node[circle, draw=black, fill=black, inner sep=0mm, minimum size=1.6mm, label={}](v_1){};	
			\draw (2.5,-1.732) node[circle, draw=black, fill=black, inner sep=0mm, minimum size=1.6mm, label={}] (v_2){};
			\draw (-1,-1.732) node[circle, draw=black, fill=black, inner sep=0mm, minimum size=1.6mm, label={}](v_3){};	
			\draw (-2.5,-1.732) node[circle, draw=black, fill=black, inner sep=0mm, minimum size=1.6mm, label={}] (v_4){};
			\draw (v)--(v_1);
			\draw (v)--(v_2);
			\draw (v)--(v_3);
			\draw (v)--(v_4);
			\draw (v)--(0.9,1.2);
			\draw (v)--(-0.9,1.2);
			\draw (v_1)--(0.5,-3);
			\draw (v_1)--(1.5,-3);
			\draw (0.5,-3)--(1.5,-3);
			\draw (v_2)--(2,-3);
			\draw (v_2)--(3,-3);
			\draw (2,-3)--(3,-3);
			\draw (v_3)--(-0.5,-3);
			\draw (v_3)--(-1.5,-3);
			\draw (-0.5,-3)--(-1.5,-3);
			\draw (v_4)--(-2,-3);
			\draw (v_4)--(-3,-3);
			\draw (-2,-3)--(-3,-3);
			\draw (0, -1.732) node {$\dots$};
			\draw (0,-4) node {(RC6)};
	\end{tikzpicture}
\end{minipage}
\hfill
% (RC7)
\begin{minipage}[b]{0.31\textwidth}
	\centering
	\begin{tikzpicture}[scale=.55]			
			\draw (0,0) node[circle, draw=black, fill=black, inner sep=0mm, minimum size=1.6mm, label={above right: $v_2$}] (v_2){};
			\draw (-3.2,0) node[circle, draw=black, fill=black, inner sep=0mm, minimum size=1.6mm, label={left: $v_1$}] (v_1){};
			\draw (3.2,0) node[circle, draw=black, fill=black, inner sep=0mm, minimum size=1.6mm, label={right: $v_3$}] (v_3){};
			\draw (0.9,-1.6) node[circle, draw=black, fill=black, inner sep=0mm, minimum size=1.6mm, label={}](u_21){};	
			\draw (-0.9,-1.6) node[circle, draw=black, fill=black, inner sep=0mm, minimum size=1.6mm, label={}](u_22){};
			\draw (2.3,-1.6) node[circle, draw=black, fill=black, inner sep=0mm, minimum size=1.6mm, label={}] (u_31){};
			\draw (4.1,-1.6) node[circle, draw=black, fill=black, inner sep=0mm, minimum size=1.6mm, label={}] (u_32){};
			\draw (-2.3,-1.6) node[circle, draw=black, fill=black, inner sep=0mm, minimum size=1.6mm, label={}] (u_11){};
			\draw (-4.1,-1.6) node[circle, draw=black, fill=black, inner sep=0mm, minimum size=1.6mm, label={}] (u_12){};
			\draw (v_1)--(v_2)--(v_3);
			\draw (v_2)--(0,1);
			\draw (v_3)--(2.5,0.9);
			\draw (v_3)--(3.9,0.9);
			\draw (v_1)--(-2.5,0.9);
			\draw (v_1)--(-3.9,0.9);
			\draw (v_2)--(u_21);
			\draw (v_2)--(u_22);
			\draw (v_3)--(u_31);
			\draw (v_3)--(u_32);
			\draw (v_1)--(u_11);
			\draw (v_1)--(u_12);
			\draw (u_21)--(0.45,-2.8);
			\draw (u_21)--(1.35,-2.8);
			\draw (0.45,-2.8)--(1.35,-2.8);
			\draw (u_22)--(-0.45,-2.8);
			\draw (u_22)--(-1.35,-2.8);
			\draw (-0.45,-2.8)--(-1.35,-2.8);
			\draw (u_31)--(1.85,-2.8);
			\draw (u_31)--(2.75,-2.8);
			\draw (1.85,-2.8)--(2.75,-2.8);
			\draw (u_32)--(3.65,-2.8);
			\draw (u_32)--(4.55,-2.8);
			\draw (3.65,-2.8)--(4.55,-2.8);
			\draw (u_11)--(-1.85,-2.8);
			\draw (u_11)--(-2.75,-2.8);
			\draw (-1.85,-2.8)--(-2.75,-2.8);
			\draw (u_12)--(-3.65,-2.8);
			\draw (u_12)--(-4.55,-2.8);
			\draw (-3.65,-2.8)--(-4.55,-2.8);
			\draw (0, -1.6) node {$\dots$};
			\draw (3.2, -1.6) node {$\dots$};
			\draw (-3.2, -1.6) node {$\dots$};
			\draw (0,-4) node {(RC7)};
	\end{tikzpicture}
\end{minipage}

\vspace{1.2em}
% (RC8)
\begin{minipage}[b]{\textwidth}
	\centering
	\begin{tikzpicture}[scale=.55]		
			\draw (1.6,0) node[circle, draw=black, fill=black, inner sep=0mm, minimum size=1.6mm, label={above right: $v_3$}] (v_3){};
			\draw (-1.6,0) node[circle, draw=black, fill=black, inner sep=0mm, minimum size=1.6mm, label={below: $v_2$}] (v_2){};
			\draw (-4.8,0) node[circle, draw=black, fill=black, inner sep=0mm, minimum size=1.6mm, label={left: $v_1$}] (v_1){};
			\draw (4.8,0) node[circle, draw=black, fill=black, inner sep=0mm, minimum size=1.6mm, label={right: $v_4$}] (v_4){};
			\draw (2.5,-1.6) node[circle, draw=black, fill=black, inner sep=0mm, minimum size=1.6mm, label={}](u_31){};	
			\draw (0.7,-1.6) node[circle, draw=black, fill=black, inner sep=0mm, minimum size=1.6mm, label={}](u_32){};
			\draw (3.9,-1.6) node[circle, draw=black, fill=black, inner sep=0mm, minimum size=1.6mm, label={}] (u_41){};
			\draw (5.7,-1.6) node[circle, draw=black, fill=black, inner sep=0mm, minimum size=1.6mm, label={}] (u_42){};
			\draw (-3.9,-1.6) node[circle, draw=black, fill=black, inner sep=0mm, minimum size=1.6mm, label={}] (u_11){};
			\draw (-5.7,-1.6) node[circle, draw=black, fill=black, inner sep=0mm, minimum size=1.6mm, label={}] (u_12){};
			\draw (v_1)--(v_2)--(v_3)--(v_4);
			\draw (v_3)--(1.6,1);
			\draw (v_4)--(4.1,0.9);
			\draw (v_4)--(5.5,0.9);
			\draw (v_1)--(-4.1,0.9);
			\draw (v_1)--(-5.5,0.9);
			\draw (v_2)--(-0.9,1);
			\draw (v_2)--(-2.3,1);
			\draw (-0.9,1)--(-2.3,1);
			\draw (v_4)--(u_41);
			\draw (v_4)--(u_42);
			\draw (v_3)--(u_31);
			\draw (v_3)--(u_32);
			\draw (v_1)--(u_11);
			\draw (v_1)--(u_12);
			\draw (u_31)--(2.05,-2.8);
			\draw (u_31)--(2.95,-2.8);
			\draw (2.05,-2.8)--(2.95,-2.8);
			\draw (u_32)--(1.15,-2.8);
			\draw (u_32)--(0.25,-2.8);
			\draw (1.15,-2.8)--(0.25,-2.8);
			\draw (u_41)--(3.45,-2.8);
			\draw (u_41)--(4.35,-2.8);
			\draw (3.45,-2.8)--(4.35,-2.8);
			\draw (u_42)--(5.25,-2.8);
			\draw (u_42)--(6.15,-2.8);
			\draw (5.25,-2.8)--(6.15,-2.8);
			\draw (u_11)--(-3.45,-2.8);
			\draw (u_11)--(-4.35,-2.8);
			\draw (-3.45,-2.8)--(-4.35,-2.8);
			\draw (u_12)--(-5.25,-2.8);
			\draw (u_12)--(-6.15,-2.8);
			\draw (-5.25,-2.8)--(-6.15,-2.8);
			\draw (1.6, -1.6) node {$\dots$};
			\draw (4.8, -1.6) node {$\dots$};
			\draw (-4.8, -1.6) node {$\dots$};
			\draw (0,-4) node {(RC8)};
	\end{tikzpicture}
\end{minipage}
	\caption{examples of (RC2)-(RC8)}
\end{figure}
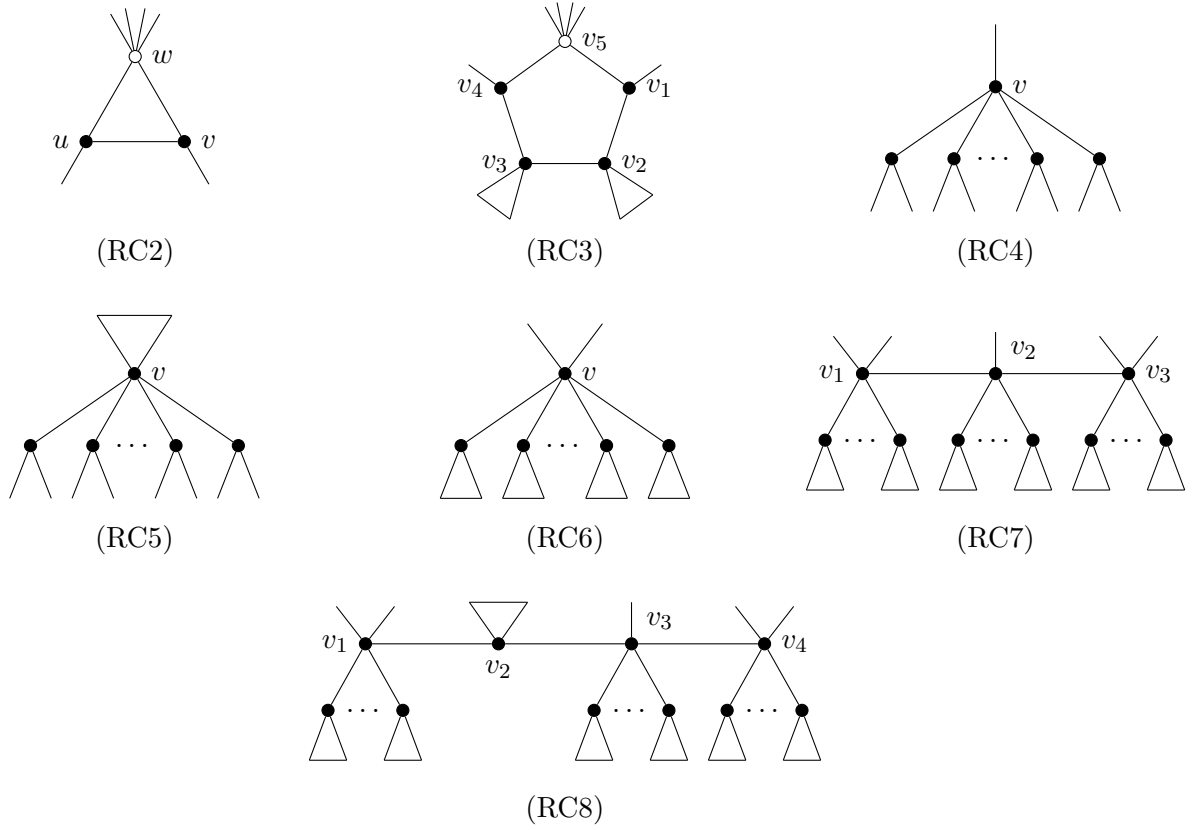

Note that (FORB) of Definition~\ref{d1} in particular implies that $k-d_G+d_{(S-B)}\ge2$ for all $v\in V(S)$. So in the proofs of Lemmas~\ref{delta} to~\ref{special} below, the case $|I|=1$ in (FORB) is implied by (FIX). Therefore by Lemma~\ref{DP2}, we only need to check the cases for (FIX) and for $|I|=2$ in (FORB) in order to prove Lemmas~\ref{delta} to~\ref{special}. Lemmas~\ref{delta} and~\ref{3-face} are from~\cite{CCFHMM22}, we include the proofs here for completeness.

\begin{lemma}(\cite{CCFHMM22})\label{delta}
(RC1) is $(\mathcal{F},4)$-boundary-reducible.
\end{lemma}
\begin{proof}
Let $S$ be the graph induced by a $2^-$-vertex $v$, and set the boundary $B=\emptyset$. (FIX) holds since $S$ has only the one vertex $v$. (FORB) also holds since there is still an available color for $v$.
\end{proof}

\begin{lemma}(\cite{CCFHMM22})\label{3-face}
(RC2) is $(\mathcal{F},4)$-boundary-reducible.
\end{lemma}

\begin{proof}
Let $S=[uvw]$ be a $3$-face with $d(u)=d(v)=3$, and set the boundary $B=\{w\}$.

(FIX): Each of $u,v$ has two available colors. Fix a color $\phi(u)$ and choose an available color for $\phi(v)$ that is not $\phi(u)$ to extend the coloring. Fixing a color for $v$ is symmetric.

(FORB): Let $I\subset V(S)\setminus B$ with $|I|=2$. Then $I$ is not $\mathcal{F}$-forbidding, since connecting a new vertex to both $u$ and $v$ creates a $4$-cycle.
\end{proof}

\begin{lemma}\label{5-face}
(RC3) is $(\mathcal{F},4)$-boundary-reducible.
\end{lemma}
\begin{proof}
Let $f=[v_1v_2v_3v_4v_5]$ be a $5$-face with $d(v_1)=d(v_4)=3$, $d(v_2)=d(v_3)=4$, and each of $v_2,v_3$ on a $3$-face $f_1$ and $f_2$, respectively. Note that $f_1$ and $f_2$ may coincide. Let $S$ be the graph on $V(f)\cup V(f_1)\cup V(f_2)$ and set the boundary $B=\{v_5\}\cup\big((V(f_1)\cup V(f_2))\setminus V(f)\big)$. For each $i\in[4]$, $v_i$ has at least two available colors.

(FIX): If $v_1$ is fixed, then we greedily color $v_2,v_3,v_4$ in order. If $v_2$ is fixed, then we greedily color $v_1,v_3,v_4$ in order. Fixing a color for $v_3$ or $v_4$ is symmetric.

(FORB): Let $I\subset V(S)\setminus B$ with $|I|=2$. Then $I$ is not $\mathcal{F}$-forbidding, since connecting a new vertex to any pair of vertices in $V(S)\setminus B$ always creates either a $4$-cycle or intersecting triangles.
\end{proof}

\begin{lemma}\label{3+1}
(RC4) is $(\mathcal{F},4)$-boundary-reducible.
\end{lemma}

\begin{proof}
Let $v$ be a $3^+$-vertex with at least $d(v)-1$ neighbors of degree $3$. Let $A$ be the set of $3$-vertices in $N(v)$. Let $S$ be the graph on $N(v)\cup\{v\}$ and set the boundary $B=N(v)\setminus A$. Note that $v$ has at least three available colors and each vertex in $A$ has at least two available colors.

(FIX): By Lemma~\ref{3-face} any two $3$-neighbors of $v$ cannot be adjacent. So $S-B$ is a star. Fixing a color of an arbitrary vertex of $S-B$ extends to all of $S-B$, since all vertices have at least two available colors.

(FORB): Let $I\subset V(S)\setminus B$ with $|I|=2$. Then only subsets of $\{v,a\}$ are $\mathcal{F}$-forbidding, where $a\in A$. In this case we can greedily color $a$, then $v$, then all vertices in $A-a$ in order.
\end{proof}

\begin{lemma}\label{3+2}
(RC5) is $(\mathcal{F},4)$-boundary-reducible.
\end{lemma}

\begin{proof}
Let $v$ be a $3^+$-vertex on a $3$-face $f=[uvw]$ with at least $d(v)-2$ neighbors of degree $3$. Let $A$ be the set of $3$-vertices in $N(v)$. Let $S$ be the graph on $N(v)\cup\{v\}$ and set the boundary $B=N(v)\setminus A$. By Lemma~\ref{3-face} any two $3$-neighbors of $v$ cannot be adjacent. So $S-B$ is a star. Observe that each vertex $x$ in $V(S)\setminus B$ has at least two available colors.

(FIX): Since $S-B$ is a star, fixing a color of an arbitrary vertex of $S-B$ extends to all of $S-B$, as all vertices have at least two available colors.

(FORB): Let $I\subset V(S)\setminus B$ with $|I|=2$. Then $I$ is not $\mathcal{F}$-forbidding, since connecting a new vertex to any pair of vertices in $V(S)\setminus B$ always creates either a $4$-cycle or intersecting triangles.
\end{proof}

\begin{lemma}\label{3+3}
(RC6) is $(\mathcal{F},4)$-boundary-reducible.
\end{lemma}

\begin{proof}
Let $v$ be a $3^+$-vertex with at least $d(v)-2$ pendent $3$-faces. By Lemma~\ref{3-face}, each pendent $3$-face of $v$ is a $(3,4^+,4^+)$-face. Let $A=\{x:\ x\in V(f^*), \text{where } f^* \text{ is a pendent }3\text{-face of } v\}$.  Let $A_1$ be the set of all $3$-vertices in $A$. Let $S$ be the graph on $N(v)\cup\{v\}\cup A$ and set the boundary $B=(N(v)\cup A)\setminus A_1$. Note that $S-B$ is a star, and each vertex $x$ in $V(S)\setminus B$ has at least two available colors.

(FIX): By Lemma~\ref{3-face} any two $3$-neighbors of $v$ cannot be adjacent. So $S-B$ is a star. Fixing a color of an arbitrary vertex of $S-B$ extends to all of $S-B$, as all vertices have at least two available colors.

(FORB): Let $I\subset V(S)\setminus B$ with $|I|=2$. Then $I$ is not $\mathcal{F}$-forbidding, since connecting a new vertex to any pair of vertices in $V(S)\setminus B$ always creates either a $4$-cycle or intersecting triangles.
\end{proof}

\begin{lemma}\label{334p}
(RC7) is $(\mathcal{F},4)$-boundary-reducible.
\end{lemma}

\begin{proof}
Let $P=v_1v_2v_3$ be a $3$-path with $d(v_1)\ge3$, $d(v_2)\ge3$, $d(v_3)\ge4$, and $v_i$ having at least $d(v_i)-3$ pendent $3$-neighbors in $N(v_i)\setminus V(P)$ for each $i\in[3]$. By Lemma~\ref{3-face}, each pendent $3$-face of $v_i$ is a $(3,4^+,4^+)$-face, where $i\in[3]$. Let $A=\{x:\ x\in V(f^*),  f^* \text{ is a pendent }3\text{-face of a vertex } v\in V(P)\}$.  Let $A_1$ be the set of all $3$-vertices in $A$. Let $S$ be the graph on $N(v_1)\cup N(v_2)\cup N(v_3)\cup A$ and set the boundary $B=\big(\bigcup_{i=1}^3 N(v_i)\cup A\big)\setminus (A_1\cup V(P))$. Note that each vertex in $A_1\cup\{v_1,v_3\}$ has at least two available colors, and $v_2$ has at least three available colors.

(FIX): By Lemma~\ref{3-face} any two vertices in $A_1$ cannot be adjacent. So $S-B$ is a star. Fixing a color of an arbitrary vertex of $S-B$ extends to all of $S-B$, since all vertices have at least two available colors.

(FORB): Let $I\subset V(S)\setminus B$ with $|I|=2$. By symmetry, $I$ may be taken to be one of the pairs $\{v_1,v_2\}$, $\{a_1,a_2\}$, $\{a_1,v_3\}$, $\{a_1,a_3\}$, where $a_i\in A_i$ for each $i\in[3]$. Recall that $S-B$ is a tree. If $I=\{v_1,v_2\}$, then we pick $v_1$ as the root and perform a breadth-first search (BFS) of $S-B$. Since each vertex of $S-B-v_1$ has at least two available colors, we can color all vertices of $S-B$ from the root $v_1$ in BFS order. If $I=\{x,y\}\in\{\{a_1,a_2\},\{a_1,v_3\},\{a_1,a_3\}\}$, then there is a unique path $P'$ connecting $x$ and $y$ in $S-B$, and $v_2\in V(P')$. We pick $v_2$ as the root and perform a BFS of $P'$. Each leaf of $P'$ has at least one available color and each other vertex of $P'$ has at least two available colors, so we can color all vertices of $P'$ from the leaves to the root in reverse-BFS order. Since each component $C$ of $S-B-V(P')$ has exactly one neighbor $c$ in $P'$, 
we pick $c$ as the root of $C\cup \{c\}$ and perform a breadth-first search of $C\cup \{c\}$. Since each vertex of $C$ has at least two available colors, we can color all vertices of $C\cup \{c\}$ from the root in BFS order.
\end{proof}

\begin{lemma}\label{3433p}
(RC8) is $(\mathcal{F},4)$-boundary-reducible.
\end{lemma}

\begin{proof}
Let $P=v_1v_2v_3v_4$ be a $4$-path with $d(v_1),d(v_3),d(v_4)\ge3$, $d(v_2)=4$, and $v_2$ on a $3$-face, where for each $i\in\{1,3,4\}$, $v_i$ has at least $d(v_i)-3$ pendent $3$-neighbors in $N(v_i)\setminus V(P)$. By Lemma~\ref{3-face}, each pendent $3$-face of $v_i$ is a $(3,4^+,4^+)$-face, where $i\in\{1,3,4\}$. Let $A=\{x:\ x\in V(f^*),  f^* \text{ is a pendent }3\text{-face of a vertex } v\in V(P)\}$.  Let $A_1$ be the set of all $3$-vertices in $A$. Let $S$ be the graph on $\big(\bigcup_{i=1}^4 N(v_i)\big)\cup A$ and set the boundary $B=\big(\bigcup_{i=1}^4 N(v_i)\cup A\big)\setminus (A_1\cup V(P))$. Note that $v_3$ has at least three available colors and each other vertex in $V(S)\setminus B$ has at least two available colors.

(FIX): By Lemma~\ref{3-face} any two vertices in $A_1$ cannot be adjacent. So $S-B$ is a star. Fixing a color of an arbitrary vertex of $S-B$ extends to all of $S-B$, since all vertices have at least two available colors.

(FORB): Let $I=\{x,y\}\subset V(S)\setminus B$ with $|I|=2$. Recall that $S-B$ is a tree. If $v_3\in I$, say $v_3=x$, then we pick $y$ as the root and perform a BFS of $S-B$; since each vertex of $S-B-y$ has at least two available colors, we can color all vertices of $S-B$ from the root $y$ in BFS order. So we may assume $v_3\notin I$. If $dist(x,y)\le2$, then $I$ is not $\mathcal{F}$-forbidding, since connecting a new vertex to any pair of vertices in $V(S)\setminus B$ always creates either a $4$-cycle or intersecting triangles. So $dist(x,y)\ge3$. Since each component of $(S-B)-v_3$ has diameter at most two, the unique path $P'$ connecting $x$ and $y$ in $S-B$ must contain $v_3$. We pick $v_3$ as the root and perform a BFS of $P'$. Each leaf of $P'$ has at least one available color, $v_3$ has at least three available colors, and each other vertex of $P'$ has at least two available colors, so we can color all vertices of $P'$ from the leaves to the root $v_3$ in reverse-BFS order. Since each component $C$ of $S-B-V(P')$ has exactly one neighbor $c$ in $P'$,
we pick $c$ as the root of $C\cup \{c\}$ and perform a breadth-first search of $C\cup \{c\}$. Since each vertex of $C$ has at least two available colors, we can color all vertices of $C\cup \{c\}$ from the root in BFS order.
\end{proof}

\begin{figure}[htbp]
	\centering
	
	% $k_1$-vertex
	\begin{minipage}[b]{0.32\textwidth}
		\centering
		\begin{tikzpicture}[scale=.65]
			\begin{scope}[xshift=4cm]			
				\draw (0,0) node[circle, draw=black, fill=black, inner sep=0mm, minimum size=1.6mm, label={right: $v$}] (v){};
				\draw (1,-1.6) node[circle, draw=black, fill=black, inner sep=0mm, minimum size=1.6mm, label={}](v_1){};	
				\draw (-1,-1.6) node[circle, draw=black, fill=black, inner sep=0mm, minimum size=1.6mm, label={}] (v_2){};
				\draw (v)--(v_1);
				\draw (v)--(v_2);
				\draw (v)--(-1.6,0);
				\draw (v)--(0.9,1.4);
				\draw (v)--(-0.9,1.4);
				\draw (0.9,1.4)--(-0.9,1.4);
				\draw (v_1)--(0.35,-2.8);
				\draw (v_1)--(1.65,-2.8);
				\draw (0.35,-2.8)--(1.65,-2.8);
				\draw (v_2)--(-0.35,-2.8);
				\draw (v_2)--(-1.65,-2.8);
				\draw (-0.35,-2.8)--(-1.65,-2.8);
				\draw (0, -1.6) node {$\dots$};
				\draw (0,-4) node {$k_1$-vertex};
			\end{scope}
		\end{tikzpicture}
	\end{minipage}
	\hfill
	% $k_2$-vertex
\begin{minipage}[b]{0.32\textwidth}
	\centering
	\begin{tikzpicture}[scale=.65]			
			\draw (0,0) node[circle, draw=black, fill=black, inner sep=0mm, minimum size=1.6mm, label={right: $v$}] (v){};
			\draw (1,-1.6) node[circle, draw=black, fill=black, inner sep=0mm, minimum size=1.6mm, label={}](v_1){};	
			\draw (-1,-1.6) node[circle, draw=black, fill=black, inner sep=0mm, minimum size=1.6mm, label={}] (v_2){};
			\draw (v)--(v_1);
			\draw (v)--(v_2);
			\draw (v)--(0,1.6);
			\draw (v)--(0.9,1.4);
			\draw (v)--(-0.9,1.4);
			\draw (v_1)--(0.35,-2.8);
			\draw (v_1)--(1.65,-2.8);
			\draw (0.35,-2.8)--(1.65,-2.8);
			\draw (v_2)--(-0.35,-2.8);
			\draw (v_2)--(-1.65,-2.8);
			\draw (-0.35,-2.8)--(-1.65,-2.8);
			\draw (0, -1.6) node {$\dots$};
			\draw (0,-4) node {$k_2$-vertex};
	\end{tikzpicture}
\end{minipage}
\hfill
	% $4_3$-vertex
\begin{minipage}[b]{0.32\textwidth}
	\centering
	\begin{tikzpicture}[scale=.65]			
		\draw (0,-0.5) node[circle, draw=black, fill=black, inner sep=0mm, minimum size=1.6mm, label={right: $v$}] (v){};
		\draw (v)--(1,-2.1);
		\draw (v)--(-1,-2.1);
		\draw (v)--(0.9,0.9);
		\draw (v)--(-0.9,0.9);
		\draw (0.9,0.9)--(-0.9,0.9);
		\draw (0,-4) node {$4_3$-vertex};
	\end{tikzpicture}
\end{minipage}
\hfill
\caption{$k_1$-, $k_2$- and $4_3$-vertices}
\end{figure}
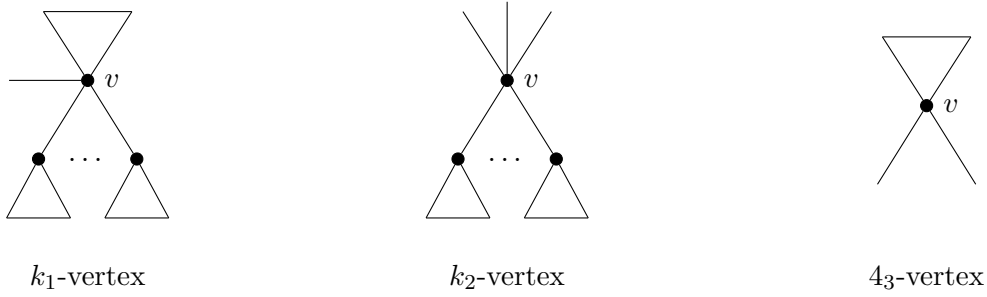

For convenience, we call $f$ \emph{a pendent $5$-face} of $v$ if $f$ is adjacent to a $3$-face containing $v$ and $v$ is not on $f$; in this case we call $v$ \emph{a weakly incident vertex} of $f$. A $5^+$-vertex $v$ is \emph{special} if $v$ has at most $d(v)-5$ pendent $3$-faces. A $4^+$-vertex $v$ is \emph{heavy} if $v$ is a $5^+$-vertex with exactly $d(v)-4$ pendent $3$-faces or a $4$-vertex with no incident $3$-face and no pendent $3$-face. A $4^+$-vertex $v$ is \emph{rich} if $v$ is a special $5^+$-vertex or a heavy $4^+$-vertex.

For $k\ge4$, a $k$-vertex $v$ is a \emph{$k_1$-vertex} if $v$ is incident to a $3$-face and has $d(v)-3$ pendent $3$-faces, and a \emph{$k_2$-vertex} if $v$ is incident to no $3$-face and has $d(v)-3$ pendent $3$-faces. A $4$-vertex $v$ is a \emph{$4_3$-vertex} if $v$ is incident to a $3$-face and has no pendent $3$-face.

Let $f$ be a $5$-face with exactly one $3$-vertex in $G$. We call $f$ \emph{poor} if $f$ satisfies one of the following conditions: (i) $f$ is incident to two $4_1$-vertices, one $4_3$-vertex, and one $4^+$-vertex $v$ with $d(v)-4$ pendent $3$-faces; (ii) $f$ is incident to two $4_3$-vertices, one $4_1$-vertex, and one other $4^+$-vertex $v$ with $d(v)-3$ pendent $3$-faces.

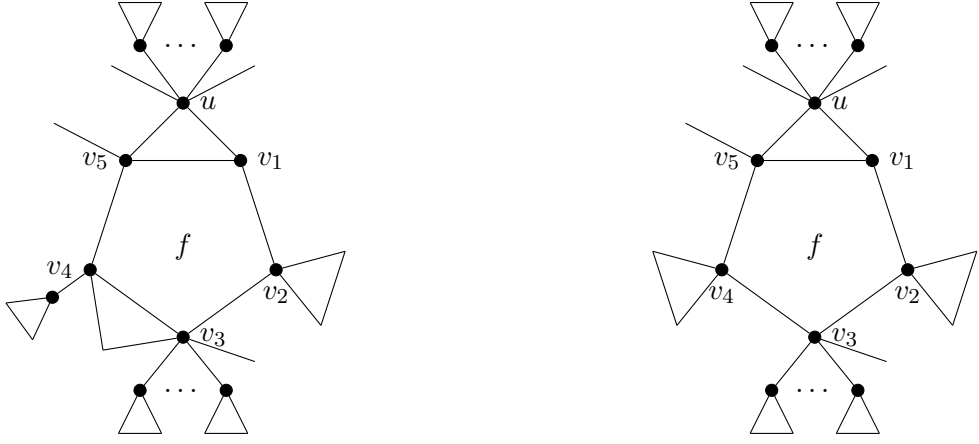
\begin{figure}[htbp]
	\centering
	
	% (i)
	\begin{minipage}[b]{0.48\textwidth}
		\centering
		\begin{tikzpicture}[scale=.38]
			\begin{scope}[xshift=4cm]			
				\draw (2,0) node[circle, draw=black, fill=black, inner sep=0mm, minimum size=1.6mm, label={right: $v_1$}] (v_1){};
				\draw (3.236,-3.804) node[circle, draw=black, fill=black, inner sep=0mm, minimum size=1.6mm, label={below: $v_2$}](v_2){};	
				\draw (0,-6.15) node[circle, draw=black, fill=black, inner sep=0mm, minimum size=1.6mm, label={right: $v_3$}] (v_3){};
				\draw (-3.236,-3.804) node[circle, draw=black, fill=black, inner sep=0mm, minimum size=1.6mm, label={left: $v_4$}](v_4){};	
				\draw (-2,0) node[circle, draw=black, fill=black, inner sep=0mm, minimum size=1.6mm, label={left: $v_5$}] (v_5){};
				\draw (0,2) node[circle, draw=black, fill=black, inner sep=0mm, minimum size=1.6mm, label={right: $u$}] (u){};
				\draw (1.5,4) node[circle, draw=black, fill=black, inner sep=0mm, minimum size=1.6mm, label={}] (u_1){};
				\draw (-1.5,4) node[circle, draw=black, fill=black, inner sep=0mm, minimum size=1.6mm, label={}] (u_2){};
				\draw (1.5,-8) node[circle, draw=black, fill=black, inner sep=0mm, minimum size=1.6mm, label={}] (u_31){};
				\draw (-1.5,-8) node[circle, draw=black, fill=black, inner sep=0mm, minimum size=1.6mm, label={}] (u_32){};
				\draw (-4.55,-4.76) node[circle, draw=black, fill=black, inner sep=0mm, minimum size=1.6mm, label={}] (u_41){};
				\draw (v_1)--(v_2)--(v_3)--(v_4)--(v_5)--(v_1);
				\draw (u)--(v_1);
				\draw (u)--(v_5);
				\draw (-4.5,1.3)--(v_5);
				\draw (u)--(2.5,3.3);
				\draw (u)--(1.5,4);
				\draw (1.5,4)--(0.75,5.5);
				\draw (1.5,4)--(2.25,5.5);
				\draw (0.75,5.5)--(2.25,5.5);
				\draw (u)--(-2.5,3.3);
				\draw (u)--(-1.5,4);
				\draw (-1.5,4)--(-0.75,5.5);
				\draw (-1.5,4)--(-2.25,5.5);
				\draw (-0.75,5.5)--(-2.25,5.5);
				\draw (v_2)--(5.64,-3.16);
				\draw (v_2)--(4.80,-5.74);
				\draw (5.64,-3.16)--(4.80,-5.74);
				\draw (v_4)--(-4.55,-4.76);
				\draw (-4.55,-4.76)--(-6.17,-4.96);
				\draw (-4.55,-4.76)--(-5.24,-6.24);
				\draw (-6.17,-4.96)--(-5.24,-6.24);
				\draw (v_3)--(1.5,-8);
				\draw (1.5,-8)--(0.75,-9.5);
				\draw (1.5,-8)--(2.25,-9.5);
				\draw (0.75,-9.5)--(2.25,-9.5);
				\draw (v_3)--(-1.5,-8);
				\draw (-1.5,-8)--(-0.75,-9.5);
				\draw (-1.5,-8)--(-2.25,-9.5);
				\draw (-0.75,-9.5)--(-2.25,-9.5);
				\draw (v_3)--(-2.79,-6.6);
				\draw (v_4)--(-2.79,-6.6);
				\draw (v_3)--(2.5,-7);
				\draw (0, -8) node {$\dots$};
				\draw (0, 4) node {$\dots$};
				\draw (0, -3) node {$f$};
			\end{scope}
		\end{tikzpicture}
	\end{minipage}
	\hfill
	% (ii)
	\begin{minipage}[b]{0.48\textwidth}
		\centering
		\begin{tikzpicture}[scale=.38]			
			\draw (2,0) node[circle, draw=black, fill=black, inner sep=0mm, minimum size=1.6mm, label={right: $v_1$}] (v_1){};
			\draw (3.236,-3.804) node[circle, draw=black, fill=black, inner sep=0mm, minimum size=1.6mm, label={below: $v_2$}](v_2){};	
			\draw (0,-6.15) node[circle, draw=black, fill=black, inner sep=0mm, minimum size=1.6mm, label={right: $v_3$}] (v_3){};
			\draw (-3.236,-3.804) node[circle, draw=black, fill=black, inner sep=0mm, minimum size=1.6mm, label={below: $v_4$}](v_4){};	
			\draw (-2,0) node[circle, draw=black, fill=black, inner sep=0mm, minimum size=1.6mm, label={left: $v_5$}] (v_5){};
			\draw (0,2) node[circle, draw=black, fill=black, inner sep=0mm, minimum size=1.6mm, label={right: $u$}] (u){};
			\draw (1.5,4) node[circle, draw=black, fill=black, inner sep=0mm, minimum size=1.6mm, label={}] (u_1){};
			\draw (-1.5,4) node[circle, draw=black, fill=black, inner sep=0mm, minimum size=1.6mm, label={}] (u_2){};
			\draw (1.5,-8) node[circle, draw=black, fill=black, inner sep=0mm, minimum size=1.6mm, label={}] (u_31){};
			\draw (-1.5,-8) node[circle, draw=black, fill=black, inner sep=0mm, minimum size=1.6mm, label={}] (u_32){};
			\draw (v_1)--(v_2)--(v_3)--(v_4)--(v_5)--(v_1);
			\draw (u)--(v_1);
			\draw (u)--(v_5);
			\draw (-4.5,1.3)--(v_5);
			\draw (u)--(2.5,3.3);
			\draw (u)--(1.5,4);
			\draw (1.5,4)--(0.75,5.5);
			\draw (1.5,4)--(2.25,5.5);
			\draw (0.75,5.5)--(2.25,5.5);
			\draw (u)--(-2.5,3.3);
			\draw (u)--(-1.5,4);
			\draw (-1.5,4)--(-0.75,5.5);
			\draw (-1.5,4)--(-2.25,5.5);
			\draw (-0.75,5.5)--(-2.25,5.5);
			\draw (v_2)--(5.64,-3.16);
			\draw (v_2)--(4.80,-5.74);
			\draw (5.64,-3.16)--(4.80,-5.74);
			\draw (v_4)--(-5.64,-3.16);
			\draw (v_4)--(-4.80,-5.74);
			\draw (-5.64,-3.16)--(-4.80,-5.74);
			\draw (v_3)--(1.5,-8);
			\draw (1.5,-8)--(0.75,-9.5);
			\draw (1.5,-8)--(2.25,-9.5);
			\draw (0.75,-9.5)--(2.25,-9.5);
			\draw (v_3)--(-1.5,-8);
			\draw (-1.5,-8)--(-0.75,-9.5);
			\draw (-1.5,-8)--(-2.25,-9.5);
			\draw (-0.75,-9.5)--(-2.25,-9.5);
			\draw (v_3)--(2.5,-7);
			\draw (0, -8) node {$\dots$};
			\draw (0, 4) node {$\dots$};
			\draw (0, -3) node {$f$};
		\end{tikzpicture}
	\end{minipage}
	\hfill
	\caption{two typical structures in (RC9), $f$ in the left one is a poor 5-face satisfying (i) and $f$ in the right one is a poor 5-face satisfying (ii).}
\end{figure}

\begin{lemma}\label{special}
(RC9) is $(\mathcal{F},4)$-boundary-reducible.
\end{lemma}

\begin{proof}
Let $f=[v_1v_2v_3v_4v_5]$ be a poor $5$-face with $d(v_1)=3$ that is weakly incident to no special $5^+$-vertices.

We first claim that $v_1$ must be on a $3$-face. Suppose otherwise; then neither $v_5$ nor $v_2$ can be a $4_1$-vertex by Lemma~\ref{3+2}. If $f$ satisfies (i), then $f$ is incident to two $4_1$-vertices, one $4_3$-vertex, and one $4^+$-vertex $v$ with $d(v)-4$ pendent $3$-faces; hence both $v_3,v_4$ are $4_1$-vertices and at least one of $v_2,v_5$ is a $4_3$-vertex, say $v_2$ by symmetry, contradicting Lemma~\ref{3433p} applied to the path $v_1v_2v_3v_4$. If $f$ satisfies (ii), then $f$ is incident to two $4_3$-vertices, one $4_1$-vertex, and one other $4^+$-vertex $v$ with $d(v)-3$ pendent $3$-faces; hence at least one of $v_3,v_4$ is a $4_1$-vertex, say $v_3$ by symmetry. Apply Lemma~\ref{334p} to the path $v_1v_2v_3$, then $v_2$ cannot be a $4^+$-vertex with $d(v_2)-3$ pendent $3$-faces, so $v_2$ is a $4_3$-vertex. Then one of $v_4,v_5$ is a $4^+$-vertex $v$ with $d(v)-3$ pendent $3$-faces. If $v_4$ is such a vertex, then applying Lemma~\ref{3433p} to the path $v_1v_2v_3v_4$ gives a contradiction; if $v_5$ is such a vertex, then applying Lemma~\ref{3433p} to the path $v_3v_2v_1v_5$ gives a contradiction. This proves the claim, so we may assume that $v_1$ is on a $3$-face $v_1v_5u$. By assumption $u$ is not a special $5^+$-vertex, so $u$ is a $4^+$-vertex with at least $d(u)-4$ pendent $3$-faces.

Let $A=\{x:\ x \text{ is a vertex on } f^*, \text{ where } f^* \text{ is a pendent }3\text{-face of a vertex } v\in V(f)\cup\{u\}\}$. By Lemma~\ref{3-face}, $f^*$ is a $(3,4^+,4^+)$-face. Let $A_1$ be the set of all $3$-vertices in $A$. Let $S$ be the graph on $\big(\bigcup_{i=1}^5 N(v_i)\big)\cup N(u)\cup A$, and set the boundary $B=V(S)\setminus (V(f)\cup A_1\cup\{u\})$. For each vertex $v\in V(S)\setminus B$, let $a(v)$ denote the minimum number of available colors for $v$ after the boundary $B$ has been colored; then $a(v)=4-d_G(v)+d_{S-B}(v)$. We have the following claim.

\textbf{Claim:} $a(v_1)=4$, $a(v_5)=3$, and $a(v)\ge2$ for each other vertex $v$ in $V(S)\setminus B$. Furthermore, there exists $i\in\{3,4\}$ with $a(v_i)=3$.

\emph{Proof of the Claim.} Since $u$ is a $4^+$-vertex with at least $d(u)-4$ pendent $3$-faces and $u$ has two neighbors $v_1,v_5$ in $S-B$, we have $a(u)=4-d(u)+(d(u)-4+2)=2$. Since each vertex $x_1\in A_1$ has degree three and has exactly one neighbor in $S-B$, $a(x_1)=4-3+1=2$. Since all neighbors of $v_1$ are in $S-B$, $a(v_1)=4$. By Lemma~\ref{3+2}, $v_5$ cannot be a $4^+$-vertex with $d(v_5)-3$ pendent $3$-faces, so by the definition of $f$, $v_5$ is either a $4_3$-vertex or a $5^+$-vertex with $d(v_5)-4$ pendent $3$-faces. Since $v_5$ has three neighbors $v_1,v_4,u$ in $S-B$, $a(v_5)=4-d(v_5)+(d(v_5)-4+3)=3$.

Since $v_2$ has a pendent $3$-neighbor $v_1$, $v_2$ cannot be a $4_3$-vertex. If $f$ is a poor $5$-face satisfying~(i), then $f$ is incident to two $4_1$-vertices, one $4_3$-vertex, and one $4^+$-vertex $v$ with $d(v)-4$ pendent $3$-faces; hence $v_2$ is either a $4_1$-vertex or a $4^+$-vertex with $d(v_2)-4$ pendent $3$-faces. In the former case $a(v_2)=4-4+2=2$, and one of $v_3,v_4$ is a $4_1$-vertex while the other is a $4_3$-vertex or a $4^+$-vertex $v$ with $d(v)-4$ pendent $3$-faces; the $4_1$-vertex $v_i\in\{v_3,v_4\}$ then has $a(v_i)=4-4+3=3$. In the latter case $v_5$ is a $4_3$-vertex and $v_3,v_4$ are $4_1$-vertices, and applying Lemma~\ref{3433p} to the path $v_1v_5v_4v_3$ yields a contradiction. If $f$ is a poor $5$-face satisfying~(ii), then $f$ is incident to two $4_3$-vertices, one $4_1$-vertex, and one other $4^+$-vertex $v$ with $d(v)-3$ pendent $3$-faces; hence $v_2$ is a $4_1$-vertex or a $4^+$-vertex with $d(v_2)-3$ pendent $3$-faces, so $a(v_2)=4-d(v_2)+(d(v_2)-3+1)=2$. Since $v_5$ is a $4_3$-vertex, one of $v_3,v_4$ is a $4_3$-vertex and the other, denoted $v_i$ for some $i\in\{3,4\}$, is a $4_1$-vertex or a $4^+$-vertex with $d(v_i)-3$ pendent $3$-faces. Then $a(v_i)=4-d(v_i)+(d(v_i)-3+2)=3$, and the vertex $v_j\in\{v_3,v_4\}\setminus\{v_i\}$ has $a(v_j)=4-4+2=2$. This completes the proof of the Claim.

(FIX): By Lemma~\ref{3-face} any two vertices in $A_1$ cannot be adjacent. Since $G$ contains no parallel edges and intersecting triangles, $u\notin V(f)$. So the subgraph induced by $(S-B)-\{v_1\}$ is a tree $T$. If $v_1$ is not fixed, then we pick the fixed vertex as the root and perform a BFS of $(S-B)-\{v_1\}$; since each vertex of $T$ has at least two available colors, we can color all vertices of $T$ from the root in BFS order, and finally color $v_1$ since $a(v_1)=4$ and $d(v_1)=3$. So we may assume $v_1$ is fixed. By the Claim, there exists $i\in\{3,4\}$ with $a(v_i)=3$. We pick $v_i$ as the root and perform a BFS of the path $P=(V(f)\cup\{u\})\setminus\{v_1\}$. Each leaf of $P$ has at least one available color and each other vertex of $P$ has at least two available colors, so we can color all vertices of $P$ from the leaves to the root in reverse-BFS order. Since each vertex in $A_1$ has exactly one neighbor in $P$, we can then color each vertex in $A_1$ greedily.

(FORB): Let $I=\{x,y\}$ with $x,y\in V(S)\setminus B$. If $dist(x,y)\le2$, then $I$ is not $\mathcal{F}$-forbidding, since connecting a new vertex to both vertices in $I$ always creates either a $4$-cycle or intersecting triangles. So we may assume $dist(x,y)\ge3$; then at most one vertex of $I$ lies on $f$.

If $v_1\notin I$, then there is exactly one path $P'$ connecting $x$ and $y$ in the tree $T=(S-B)-\{v_1\}$. Since each component of $T-\{v_3,v_4,v_5\}$ has diameter at most $2$, either $v_5\in V(P')$ or $v_3,v_4\in V(P')$. By the Claim, $P'$ contains a vertex $z\in\{v_3,v_4,v_5\}$ with $a(z)\ge3$. We pick $z$ as the root and perform a BFS of $P'$. Each leaf of $P'$ has at least one available color and each other vertex of $P'$ has at least two available colors, so we can color all vertices of $P'$ from the leaves to the root in reverse-BFS order. Since each component $C$ of $T-V(P')$ has exactly one neighbor $c$ in $P'$, we pick $c$ as the root of $C\cup \{c\}$ and perform a breadth-first search of $C\cup \{c\}$. Since each vertex of $C$ has at least two available colors, we can color all vertices of $C\cup \{c\}$ from the root in BFS order. Finally, we color $v_1$ greedily, since $a(v_1)=4$ and $d(v_1)=3$.

If $v_1\in I$, say $v_1=x$ by symmetry, then $y$ must be a pendent $3$-neighbor of a vertex $v_j\in\{v_3,v_4\}$. If $a(v_j)=3$, then we choose an available color $c_0$ for $v_1$ such that $c_0$ is not matched with any available color of $u$. Let $P'$ be the path connecting $v_2$ and $y$ with each internal vertex in $\{v_3,v_4\}$; note that $v_j\in V(P')$. We pick $v_j$ as the root and perform a BFS of $P'$. Each leaf of $P'$ has at least one available color and each other vertex of $P'$ has at least two available colors, so we can color all vertices of $P'$ from the leaves to the root in reverse-BFS order. For each component $C$ of $T-V(P')$, $C$ has exactly one vertex $c$ with a neighbor in $V(P')$. We pick $c$ as the root and perform a BFS of $C\cup \{c\}$, and since each vertex of $C$ has at least two available colors, we color all vertices of $C$ from the root in BFS order. 

So we may assume that $a(v_j)=2$. If $v_j=v_3$, then we first color $y,v_3,v_2$ greedily; then we choose an available color $c$ for $v_5$ such that $c$ is not matched with any available color of $u$, color $v_1,u,v_4$ in order, and finally color each vertex in $A_1-\{y\}$. If $v_j=v_4$, then we first color $y,v_4$ greedily; then we choose an available color $c$ for $v_1$ such that $c$ is not matched with any available color of $u$, color $v_5,u,v_2,v_3$ in order, and finally color each vertex in $A_1-\{y\}$.
\end{proof}

\section{\bf Discharging}
We now present a discharging procedure that completes the proof of Theorem~\ref{main}. For each vertex $v$ and each face $f$ of $G$, let $\mu(v)=d(v)-2$ and $\mu(f)=-2$ be the initial charges. By Euler's formula, $\sum_{v\in V(G)}(d(v)-2)+\sum_{f\in F(G)}(-2)=-4$. Let $\mu^*(x)$ denote the charge of $x\in V(G)\cup F(G)$ after the discharging procedure. To reach a contradiction, we prove that $\mu^*(x)\ge0$ for all $x\in V(G)\cup F(G)$.

The discharging rules are as follows.

\begin{itemize}
\item[(R1)] Every $3$-vertex sends $\frac{1}{3}$ to each incident face.

\item[(R2)] Every $4^+$-vertex sends $\frac{2}{3}$ to each incident $3$-face and $\frac{1}{3}$ to each pendent $3$-face.

\item[(R3)] Every special $5^+$-vertex sends $\frac{1}{9}$ to each pendent $5$-face.

\item[(R4)] Every $5^+$-face gets $\frac{1}{3}$ from each incident $4_1$-vertex, $\frac{5}{12}$ from each incident $4_2$-vertex or $k_1$-vertex, $\frac{4}{9}$ from each incident $4_3$-vertex, $\frac{7}{15}$ from each incident $k_2$-vertex, $\frac{1}{2}$ from each incident heavy $4^+$-vertex, and $\frac{5}{9}$ from each incident special $5^+$-vertex, where $k\ge5$.
\end{itemize}

\begin{lemma}
Each vertex has non-negative final charge.
 \end{lemma}
\begin{proof}
Let $v$ be a $k$-vertex of $G$. By Lemma~\ref{delta}, $k\ge3$. If $k=3$, then $v$ sends $\frac{1}{3}$ to each incident face by (R1), so $\mu^*(v)\ge 3-2-\frac{1}{3}\times3=0$. If $k=4$, then $v$ has at most one pendent $3$-face by Lemma~\ref{3+3}, so $v$ is a $4_1$-vertex, a $4_2$-vertex, a $4_3$-vertex, or a heavy vertex. By (R2) and (R4), $v$ sends $\frac{2}{3}$ to each incident $3$-face, $\frac{1}{3}$ to each pendent $3$-face, and to each incident $5^+$-face it sends $\frac{1}{3}$ if $v$ is a $4_1$-vertex, $\frac{5}{12}$ if $v$ is a $4_2$-vertex, $\frac{4}{9}$ if $v$ is a $4_3$-vertex, and $\frac{1}{2}$ if $v$ is heavy. So $\mu^*(v)\ge4-2-\max\{\frac{2}{3}+\frac{1}{3}\times4,\ \frac{1}{3}+\frac{5}{12}\times4,\ \frac{2}{3}+\frac{4}{9}\times3,\ \frac{1}{2}\times4\}=0$.

If $k\ge5$, then $v$ has at most $k-3$ pendent $3$-faces by Lemma~\ref{3+3}, so $v$ is a $k_1$-vertex, a $k_2$-vertex, a heavy vertex, or a special vertex. By (R2), $v$ sends $\frac{2}{3}$ to each incident $3$-face and $\frac{1}{3}$ to each pendent $3$-face. If $v$ is a $k_1$-vertex, then $v$ sends $\frac{5}{12}$ to each incident $5^+$-face by (R4), so $\mu^*(v)\ge k-2-\frac{2}{3}-\frac{1}{3}(k-3)-\frac{5}{12}(k-1)\ge0$. If $v$ is a $k_2$-vertex, then $v$ sends $\frac{7}{15}$ to each incident $5^+$-face by (R4), so $\mu^*(v)\ge k-2-\frac{1}{3}(k-3)-\frac{7}{15}k\ge0$. If $v$ is heavy, then $v$ has $k-4$ pendent $3$-faces and sends $\frac{1}{2}$ to each incident $5^+$-face by (R4); since $v$ is incident either to only $5^+$-faces or to one $3$-face and other $5^+$-faces, in either case $\mu^*(v)\ge k-2-\max\{\frac{1}{2}k+\frac{1}{3}(k-4),\ \frac{2}{3}+\frac{1}{2}(k-1)+\frac{1}{3}(k-4)\}\ge0$. If $v$ is special, then $v$ has at most $k-5$ pendent $3$-faces and sends $\frac{5}{9}$ to each incident $5^+$-face and $\frac{1}{9}$ to each pendent $5$-face by (R3) and (R4); since $v$ is incident either to a $3$-face or not, in either case $\mu^*(v)\ge k-2-\max\{\frac{2}{3}+\frac{1}{9}+\frac{5}{9}(k-1)+\frac{1}{3}(k-5),\ \frac{5}{9}k+\frac{1}{3}(k-5)\}\ge0$.
\end{proof}

Since $G$ contains no $4$-cycles, it suffices to show that each $3$-face and each $5^+$-face has non-negative final charge.

\begin{lemma}
Each $3$-face and each $6^+$-face has non-negative final charge.
 \end{lemma}
\begin{proof}
Let $f$ be a $k$-face of $G$. If $k=3$, then each $4^+$-vertex on $f$ sends $\frac{2}{3}$ to $f$ by (R2). If $v$ is a $3$-vertex on $f$, then the neighbor $w$ of $v$ not on $f$ is a $4^+$-vertex by Lemma~\ref{3+2}, and by (R1) and (R2) the pair $v,w$ sends $\frac{1}{3}+\frac{1}{3}=\frac{2}{3}$ to $f$. Therefore each vertex on $f$ guarantees at least $\frac{2}{3}$ to $f$, so $\mu^*(f)\ge-2+\frac{2}{3}\times3=0$. If $k\ge6$, then each vertex on $f$ sends at least $\frac{1}{3}$ to $f$ by (R4), so $\mu^*(f)\ge -2+\frac{1}{3}k\ge0$.
\end{proof}

\begin{lemma}
Each $5$-face has non-negative final charge.
 \end{lemma}
\begin{proof}
Let $f=[v_1v_2v_3v_4v_5]$ be a $5$-face of $G$. By Lemma~\ref{3+1}, a $3$-vertex has at most one $3$-neighbor, so $f$ is incident to at most three $3$-vertices. We consider three cases.

\textbf{Case 1:} $f$ is incident to at least two $3$-vertices. By symmetry, either $v_1,v_2$ are $3$-vertices or $v_1,v_3$ are $3$-vertices.

\textbf{Case 1.1:} $v_1,v_2$ are $3$-vertices. By Lemma~\ref{3+1}, each of $v_3,v_5$ is a $4^+$-vertex. Applying Lemma~\ref{334p} to the path $v_1v_2v_3$, $v_3$ has at most $d(v_3)-4$ pendent $3$-neighbors in $N(v_3)-v_2$. By Lemma~\ref{3-face}, $v_2$ cannot be a pendent $3$-neighbor of $v_3$, so $v_3$ has at most $d(v_3)-4$ pendent $3$-neighbors in $G$. If $d(v_4)=3$ or $v_4$ is a $4^+$-vertex with $d(v_4)-3$ pendent $3$-faces, then by Lemma~\ref{3+2} $v_3$ cannot be on a $3$-face when $d(v_3)=4$, so $v_3$ is either a $4$-vertex with no incident $3$-face or pendent $3$-face, or a $5^+$-vertex with at most $d(v_3)-4$ pendent $3$-faces; hence $v_3$ is rich. By the symmetry of $v_3$ and $v_5$, $v_5$ is also rich. By (R4), each of $v_3,v_5$ sends at least $\frac{1}{2}$ to $f$ and each of $v_1,v_2,v_4$ sends at least $\frac{1}{3}$ to $f$, so $\mu^*(f)\ge-2+\frac{1}{3}\times3+\frac{1}{2}\times2=0$. So we may assume that $v_4$ is a $4^+$-vertex with at most $d(v_4)-4$ pendent $3$-faces. Recall that each $v_i$ ($i\in\{3,5\}$) is also a $4^+$-vertex with at most $d(v_i)-4$ pendent $3$-faces. So each of $v_3,v_4,v_5$ is either a $4_3$-vertex or a rich $5^+$-vertex, which sends at least $\frac{4}{9}$ to $f$ by (R4). Hence $\mu^*(f)\ge-2+\frac{1}{3}\times2+\frac{4}{9}\times3=0$.

\textbf{Case 1.2:} $v_1,v_3$ are $3$-vertices. Then each of $v_2,v_4,v_5$ is a $4^+$-vertex, and by (R1) each of $v_1,v_3$ sends $\frac{1}{3}$ to $f$. Since $G$ contains no intersecting triangles, at most one of $v_1v_2,v_2v_3$ is on a $3$-face. We consider two subcases.

\textbf{Case 1.2.1:} Neither $v_1v_2$ nor $v_2v_3$ is on a $3$-face adjacent to $f$. Since $v_3$ is not a pendent $3$-neighbor of $v_4$, $v_4$ cannot be a $4_1$- or $5_1$-vertex by Lemma~\ref{3+2}. By symmetry, $v_5$ cannot be a $4_1$- or $5_1$-vertex. For each $i\in\{4,5\}$, by (R4) $v_i$ sends at least $\frac{5}{12}$ to $f$ if $d(v_i)=4$ and at least $\frac{7}{15}$ to $f$ if $d(v_i)\ge5$. By Lemma~\ref{3+2}, $v_2$ is not incident to a $3$-face when $d(v_2)=4$, so $v_2$ is a $4_2$-vertex or rich, which sends at least $\frac{5}{12}$ to $f$.

Suppose $d(v_4),d(v_5)\ge5$. Then by the above, $\mu^*(f)\ge-2+\frac{5}{12}+\frac{1}{3}\times2+\frac{7}{15}\times2=\frac{1}{60}>0$.

Suppose $d(v_4)=d(v_5)=4$. By Lemma~\ref{3+3}, $v_2$ has at most $d(v_2)-3$ pendent $3$-faces. If $v_2$ has at most $d(v_2)-4$ pendent $3$-faces, then $v_2$ is rich and sends at least $\frac{1}{2}$ to $f$ by (R4), so $\mu^*(f)\ge-2+\frac{1}{2}+\frac{1}{3}\times2+\frac{5}{12}\times2=0$.
If $v_2$ has $d(v_2)-3$ pendent $3$-faces, then $v_2$ sends at least $\frac{5}{12}$ to $f$ by (R4). By Lemma~\ref{3+1}, at least one of $v_1,v_3$ is a pendent $3$-neighbor of $v_2$, say $v_1$ by symmetry; then $v_5$ cannot be a $4_1$-vertex by Lemma~\ref{3+2}, so $v_5$ is a $4_3$-vertex and sends $\frac{4}{9}$ to $f$ by (R4). By Lemma~\ref{5-face}, $v_4$ cannot be incident to a $3$-face. Applying Lemma~\ref{3433p} to the path $v_1v_5v_4v_3$, $v_4$ has no pendent $3$-face, so $v_4$ is rich and sends $\frac{1}{2}$ to $f$ by (R4). Hence $\mu^*(f)\ge-2+\frac{5}{12}+\frac{1}{3}\times2+\frac{1}{2}+\frac{4}{9}=\frac{1}{36}>0$.

Suppose $d(v_4)=4$ and $d(v_5)\ge5$. If $v_4$ is incident to a $3$-face, then $v_5$ has at most $d(v_5)-4$ pendent $3$-faces by applying Lemma~\ref{3433p} to the path $v_3v_4v_5v_1$. If $v_4$ has a pendent $3$-face, then $v_5$ has at most $d(v_5)-4$ pendent $3$-faces by applying Lemma~\ref{334p} to the path $v_3v_4v_5$. Therefore $v_5$ has at most $d(v_5)-4$ pendent $3$-faces whenever $v_4$ is not rich, and by (R4) $v_5$ sends at least $\frac{1}{2}$ to $f$. If $v_4$ is rich, then $v_4$ sends $\frac{1}{2}$ to $f$ by (R4). In either case, $v_4,v_5$ send at least $\frac{5}{12}+\frac{1}{2}=\frac{11}{12}$ to $f$, so $\mu^*(f)\ge-2+\frac{5}{12}+\frac{1}{3}\times2+\frac{11}{12}=0$.

\textbf{Case 1.2.2:} Exactly one of $v_1v_2,v_2v_3$ is on a $3$-face adjacent to $f$, say $v_1v_2$ by symmetry. By Lemma~\ref{3+2}, $d(v_2)\ge5$ and $v_2$ has at most $d(v_2)-4$ pendent $3$-faces, so $v_2$ is rich and sends at least $\frac{1}{2}$ to $f$ by (R4). Thus $v_1,v_2,v_3$ send at least $\frac{1}{2}+\frac{1}{3}+\frac{1}{3}=\frac{7}{6}$ to $f$ in total, and it suffices to show that $v_4,v_5$ send at least $\frac{5}{6}$ to $f$.

Suppose $d(v_4)=d(v_5)=4$. Applying Lemma~\ref{334p} to the path $v_3v_4v_5$, $v_4$ has no pendent $3$-face, so $v_4$ is either a $4_3$-vertex or rich. If $v_4$ is a $4_3$-vertex, then $v_5$ cannot be on a $3$-face by Lemma~\ref{5-face}, so $v_5$ is a $4_2$-vertex and sends $\frac{5}{12}$ to $f$ by (R4). If $v_4$ is rich, then $v_4$ sends $\frac{1}{2}$ to $f$ and $v_5$ sends at least $\frac{1}{3}$ to $f$ by (R4). In either case, $v_4,v_5$ send at least $\min\{\frac{4}{9}+\frac{5}{12},\ \frac{1}{2}+\frac{1}{3}\}=\frac{5}{6}$ to $f$.

Suppose $d(v_5)=4$ and $d(v_4)\ge5$. Applying Lemma~\ref{334p} to the path $v_3v_4v_5$, $v_4$ has at most $d(v_4)-4$ pendent $3$-faces, so $v_4$ sends at least $\frac{1}{2}$ to $f$ by (R4). Hence $v_4,v_5$ send at least $\frac{1}{2}+\frac{1}{3}=\frac{5}{6}$ to $f$.

Suppose $d(v_4)\ge4$ and $d(v_5)\ge5$. Since $v_3$ is not a pendent $3$-neighbor of $v_4$, $v_4$ cannot be a $4_1$-vertex by Lemma~\ref{3+2}. So each of $v_4,v_5$ sends at least $\frac{5}{12}$ to $f$ by (R4), and $v_4,v_5$ send a total of $\frac{5}{12}+\frac{5}{12}=\frac{5}{6}$ to $f$.

For convenience, in the rest of the proof let $K=\{v_j:\ v_j \text{ is a } 4^+\text{-vertex with at most } d(v_j)-4 \text{ pendent } 3\text{-faces},\ 1\le j\le5\}$. If $v_1$ is a $3$-vertex on a $3$-face, say $v_1v_5$ is incident to a $3$-face by symmetry, then $v_5$ has at most $d(v_5)-4$ pendent $3$-faces by Lemma~\ref{3+2}, so $v_5\in K$. If $v_1$ is not on any $3$-face or $d(v_1)\ge4$, then applying Lemma~\ref{334p} to $v_5v_1v_2$, at least one of $v_2,v_5$ is in $K$, say $v_5$; and applying Lemma~\ref{334p} to $v_2v_3v_4$, at least one of $v_2,v_3,v_4$ is in $K$. So $|K|\ge2$; furthermore, if $|K|=2$, the two vertices in $K$ are not adjacent. If $|K|\ge3$, then $f$ gets at least $\frac{4}{9}$ from each vertex in $K$ and at least $\frac{1}{3}$ from each vertex in $V(f)-K$ by (R4), so $\mu^*(f)\ge-2+\frac{4}{9}\times3+\frac{1}{3}\times2=0$. Thus $|K|=2$. By Lemma~\ref{3+2}, each vertex $v_i$ in $V(f)-K$ has exactly $d(v_i)-3$ pendent $3$-faces.

\textbf{Case 2:} $f$ is incident to exactly one $3$-vertex, say $v_1$, so $f$ is a $(3,4^+,4^+,4^+,4^+)$-face. If $f$ is incident to at least two rich vertices, then by (R1) and (R4) $f$ gets at least $\frac{1}{2}$ from each incident rich vertex and at least $\frac{1}{3}$ from each other incident vertex, so $\mu^*(f)\ge-2+\frac{1}{2}\times2+\frac{1}{3}\times3=0$. So we may assume that $f$ is incident to at most one rich vertex.

\textbf{Case 2.1:} $f$ is incident to at least two $5^+$-vertices $v_i,v_j$ for some $i,j\in\{2,3,4,5\}$. If $v_i,v_j\in K$, then both $v_i$ and $v_j$ are rich, contrary to our assumption. If $v_i,v_j\notin K$, then by (R4) $f$ gets at least $\frac{5}{12}$ from each of $v_i,v_j$ and at least $\frac{4}{9}$ from each of at least two vertices in $K$, so $\mu^*(f)\ge-2+\frac{1}{3}+\frac{5}{12}\times2+\frac{4}{9}\times2=\frac{1}{18}>0$. If $v_i\in K$ and $v_j\notin K$, then by (R1) and (R4) $f$ gets at least $\frac{1}{2}$ from $v_i$, $\frac{5}{12}$ from $v_j$, $\frac{4}{9}$ from the vertex in $K-v_i$, and at least $\frac{1}{3}$ from each other incident vertex, so $\mu^*(f)\ge-2+\frac{1}{3}\times2+\frac{1}{2}+\frac{5}{12}+\frac{4}{9}=\frac{1}{36}>0$.

\textbf{Case 2.2:} $f$ is incident to exactly one $5^+$-vertex $v_i$ for some $i\in\{2,3,4,5\}$. If $v_i\in K$, then $v_i$ is rich, and by our assumption the $4$-vertex in $K-v_i$ is a $4_3$-vertex, which sends $\frac{4}{9}$ to $f$ by (R4). If $v_i$ has at most $d(v_i)-5$ pendent $3$-faces, then by (R1) and (R4) $v_i$ sends at least $\frac{5}{9}$ to $f$ and each vertex in $V(f)-K$ sends at least $\frac{1}{3}$ to $f$, so $\mu^*(f)\ge-2+\frac{1}{3}\times3+\frac{4}{9}+\frac{5}{9}=0$. So we may assume that $v_i$ has $d(v_i)-4$ pendent $3$-faces; then by (R4) $v_i$ sends at least $\frac{1}{2}$ to $f$. If some vertex $v_j\in\{v_2,v_3,v_4,v_5\}-K$ is a $4_2$-vertex, then $v_j$ sends $\frac{5}{12}$ to $f$ by (R4), so $\mu^*(f)\ge-2+\frac{1}{2}+\frac{4}{9}+\frac{5}{12}+\frac{1}{3}\times2=\frac{1}{36}>0$. Otherwise $f$ is a $5$-face with two $4_1$-vertices, one $4_3$-vertex, and one $5^+$-vertex $v_i$ with $d(v_i)-4$ pendent $3$-faces, that is, a poor $5$-face. By Lemma~\ref{special}, $f$ is weakly incident to a special $5^+$-vertex, so by (R1), (R3), and (R4) $f$ gets $\frac{1}{9}$ from its weakly incident special $5^+$-vertex, $\frac{1}{3}$ from each incident $4_1$-vertex or $3$-vertex, $\frac{4}{9}$ from each incident $4_3$-vertex, and at least $\frac{1}{2}$ from $v_i$. Hence $\mu^*(f)\ge-2+\frac{1}{3}\times3+\frac{4}{9}+\frac{1}{2}+\frac{1}{9}=\frac{1}{18}>0$.

If $v_i\notin K$, then each vertex in $K$ is a $4$-vertex and $v_i$ has $d(v_i)-3$ pendent $3$-faces. By our assumption, at most one vertex in $K$ is rich. If some vertex in $K$ is rich, then by (R1) and (R4) $f$ gets at least $\frac{5}{12}$ from $v_i$, $\frac{1}{2}+\frac{4}{9}$ from the two vertices in $K$, and at least $\frac{1}{3}$ from each other incident vertex, so
$\mu^*(f)\ge-2+\frac{5}{12}+\frac{1}{2}+\frac{4}{9}+\frac{1}{3}\times2=\frac{1}{36}>0$.
So we may assume that each vertex in $K$ is a $4_3$-vertex. Let $v_j$ be the vertex in $V(f)-(K\cup\{v_i,v_1\})$ for some $j\in\{2,3,4,5\}$. Then $v_j$ is either a $4_2$-vertex or a $4_1$-vertex. In the former case $v_j$ sends $\frac{5}{12}$ to $f$, so $\mu^*(f)\ge-2+\frac{5}{12}\times2+\frac{1}{3}+\frac{4}{9}\times2=\frac{1}{18}>0$.
In the latter case $f$ is a $5$-face with two $4_3$-vertices, one $4_1$-vertex, and one $5^+$-vertex $v_i$ with $d(v_i)-3$ pendent $3$-faces, that is, a poor $5$-face. By Lemma~\ref{special}, $f$ is weakly incident to a special $5^+$-vertex, so by (R1), (R3), and (R4) $f$ gets $\frac{1}{9}$ from its weakly incident special $5^+$-vertex, $\frac{1}{3}$ from each incident $4_1$-vertex or $3$-vertex, $\frac{4}{9}$ from each incident $4_3$-vertex, and at least $\frac{5}{12}$ from $v_i$. Hence $\mu^*(f)\ge-2+\frac{1}{3}\times2+\frac{5}{12}+\frac{4}{9}\times2+\frac{1}{9}=\frac{1}{12}>0$.

\textbf{Case 2.3:} $f$ is incident to no $5^+$-vertex, so $f$ is a $(3,4,4,4,4)$-face. Since $f$ is incident to at most one rich vertex and $|K|=2$, $f$ is incident to at least one $4_3$-vertex, say $v_i$ for some $i\in\{2,3,4,5\}$, and to two $4_1$- or $4_2$-vertices. If $f$ is incident to two $4_2$-vertices, then by (R1) and (R4) $f$ gets $\frac{5}{12}$ from each incident $4_2$-vertex, at least $\frac{4}{9}$ from each vertex in $K$, and $\frac{1}{3}$ from the incident $3$-vertex, so $\mu^*(f)\ge-2+\frac{5}{12}\times2+\frac{4}{9}\times2+\frac{1}{3}=\frac{1}{18}>0$. If $f$ is incident to two $4_1$-vertices, then $f$ is a $5$-face with two $4_1$-vertices, one $4_3$-vertex, and one $4$-vertex with no pendent $3$-face (say $v_j$ for some $j\in\{2,3,4,5\}$), that is, a poor $5$-face. By (R1), (R3), and (R4) $f$ gets $\frac{1}{9}$ from its weakly incident special $5^+$-vertex, $\frac{1}{3}$ from each incident $4_1$-vertex or $3$-vertex, $\frac{4}{9}$ from each incident $4_3$-vertex, and at least $\frac{4}{9}$ from $v_j$, so $\mu^*(f)\ge-2+\frac{1}{3}\times3+\frac{4}{9}\times2+\frac{1}{9}=0$. If $f$ is incident to one $4_1$-vertex and one $4_2$-vertex, then by (R1) and (R4) $f$ gets $\frac{1}{3}$ from each incident $4_1$-vertex or $3$-vertex, $\frac{5}{12}$ from each incident $4_2$-vertex, and $\frac{4}{9}$ from each incident $4_3$-vertex. Note that the vertex in $K-v_i$ is either a $4_3$-vertex or rich. In the former case $f$ is poor, and by Lemma~\ref{special} $f$ is weakly incident to a special $5^+$-vertex, so by (R3) $f$ gets $\frac{1}{9}$ from it. In the latter case, by (R4) $f$ gets at least $\frac{1}{2}$ from the incident rich vertex. Therefore, in either case $\mu^*(f)\ge-2+\frac{1}{3}\times2+\frac{5}{12}+\frac{4}{9}+\min\{\frac{4}{9}+\frac{1}{9},\ \frac{1}{2}\}=\frac{1}{36}>0$.

\textbf{Case 3:} $f$ is incident to no $3$-vertex. Recall that $|K|=2$ and the two vertices in $K$ are not adjacent, say $v_1,v_3$ by symmetry. Applying Lemma~\ref{3433p} to the path $v_2v_1v_5v_4$ or the path $v_2v_3v_4v_5$, neither $v_1$ nor $v_3$ is a $4_3$-vertex. So by (R1) and (R4) $f$ gets at least $\frac{1}{2}$ from each of $v_1,v_3$ and at least $\frac{1}{3}$ from each other incident vertex, so $\mu^*(f)\ge -2+\frac{1}{2}\times2+\frac{1}{3}\times3=0$.
\end{proof}

\end{document}